\documentclass[11pt,a4paper,reqno]{amsart}
\usepackage[margin=25mm,headheight=14pt]{geometry}
\usepackage{amsmath,amssymb,mathtools,mathrsfs}
\usepackage{xcolor,graphicx}
\usepackage{placeins}
\usepackage{float}
\usepackage{hyperref}
\hypersetup{hidelinks,pdftitle={Spiral Harmonic Products},pdfauthor={Hai Zhang and Yongsheng Zhang}}
\allowdisplaybreaks[2]
\numberwithin{equation}{section}
\newtheorem{theorem}{Theorem}[section]
\newtheorem{lemma}[theorem]{Lemma}
\newtheorem{proposition}[theorem]{Proposition}
\newtheorem{corollary}[theorem]{Corollary}
\theoremstyle{remark}
\newtheorem{remark}[theorem]{Remark}
\newcommand{\R}{\mathbb R}
\newcommand{\C}{\mathbb C}
\newcommand{\Z}{\mathbb Z}
\newcommand{\Sph}{\mathbb S}
\begin{document}
\title{Spiral Harmonic Products}
\author{Hai Zhang}
\author{Yongsheng Zhang}
\begin{abstract}
Motivated by the systematic study of spiral minimal products, we consider harmonic maps obtained by coupling two spherical eigenmaps through a profile curve in $\Sph^3$, with a doubly warped product metric on the source.
Keeping the source warps independent of the profile magnitudes leads us to two complementary problems: we reconstruct a compatible source metric from prescribed profile or source data, and we seek closed profiles after fixing the source metric.
For arbitrary prescribed warps, we homogenize the reduced Lagrangian and obtain a strongly convex conic Finsler metric whose oriented graph geodesics are exactly the harmonic profiles.
In the unit-speed inverse problem, the conserved phase momenta reduce reconstruction to a scalar magnitude equation.
This reduction yields global reconstructions from a prescribed magnitude, one source warp, or a monotone volume factor, as well as periodic families in which profiles closing on finite covers are dense.
For the fixed-metric problem, a change of time leads to an autonomous Neumann system.
We use this system to find continuous families of closed profiles and infinitely many closed graph geodesics for one fixed source metric.
For fixed even $p,q\geq 2$, every metric in this Neumann class on  $\Sph^1\times\Sph^p\times\Sph^q$ supports harmonic maps into  $\Sph^{p+q+1}$ of degree $4m$ for every $m\geq1$.
In the matched constant-warp case, these maps reduce to explicit eigenmaps.
\end{abstract}

\maketitle
\section{Introduction}
\subsection{Background and the two problems}
A map $f$ from a Riemannian manifold into the unit sphere is harmonic precisely when $\Delta f=-|df|^2f$.
Here $|df|^2$ is its energy density.
An \emph{eigenmap} is a spherical harmonic map with constant energy density.
Minimal isometric immersions are basic examples by Takahashi's theorem \cite{T66}.
For general background, see \cite{ES64}.
For the equivariant ODE viewpoint, see \cite{ER93}.

We place two eigenmaps in the two complex blocks of the target sphere and use a curve in $\Sph^3$ to control their relative sizes.
Two independent positive functions rescale the corresponding source factors.
This separation leads us to two problems.
In the first, we vary the source metric and reconstruct it from a prescribed profile, one warp, or the volume factor.
In the second, we fix the full source metric and seek closed profiles on it.

Smith's harmonic joins already separate spherical inputs through a one-dimensional profile \cite{Sm75}.
Baird and Ratto developed conservation laws for equivariant harmonic maps \cite{BR92}.
Ratto also constructed harmonic maps by deforming the domain metric, including warped-product settings \cite{Ra89a,Ra89b}.
Recent related constructions include deformation of cohomogeneity-one domain metrics and eigenmap-based framework \cite{Si26,BS26}.
Here our inverse problem instead reconstructs warped source geometry from prescribed profile data, while the fixed-metric problem asks a single source metric to support many closed profiles.

A second ingredient comes from the geometry of time-dependent Lagrangians.
By positive homogenization on the extended configuration space, we turn the reduced trajectories into geodesics of a conic Finsler metric \cite{Me16}.
This metric is auxiliary.
The source and target of the harmonic map remain Riemannian, so the construction differs from harmonic maps with a Finsler domain \cite{Mo01}.

The closest companion is the systematic study of spiral minimal products by Li and the second author \cite{LZ26}.
That work treats weighted profile geometry, integrability, and closed families.
In the induced-metric situation of Remark \ref{hlzrc}, the harmonic and minimal profile equations agree.
In general they do not: here the source warps are independent of the profile magnitudes, and the resulting equation is usually nonautonomous.

\subsection{The spiral harmonic product}
Fix eigenmaps
\begin{equation*}
      f_i:(M_i^{k_i},g_i)\longrightarrow\Sph^{2n_i+1}\subset\C^{n_i+1},
      \qquad \Delta f_i=-d_if_i,
       \qquad k_i,d_i>0.
\end{equation*}
Here $d_i=|df_i|^2$ is constant, and $\Sph^N$ is the unit round sphere.
For a profile curve $\gamma=(z_1,z_2):I\to\Sph^3$ and positive source warps $\ell_i$, set
\begin{equation*}
         G_\gamma(t,x,y)=(z_1(t)f_1(x),z_2(t)f_2(y)),
          \qquad
         g_\gamma=dt^2+\ell_1^2g_1+\ell_2^2g_2.
\end{equation*}
When $G_\gamma$ is harmonic, we call it a \emph{spiral harmonic product}.
In the principal region $z_1z_2\ne0$, express the profile as $\gamma=(ae^{is_1},be^{is_2})$ with $a,b>0$.
The profile magnitudes $a,b$ are target coefficients, whereas the source warps $\ell_1,\ell_2$ define the domain metric.
No relation $a=\ell_1$, $b=\ell_2$ is imposed.
The product is \emph{singly spiral} if one phase rotates, \emph{doubly spiral} if both do, and has \emph{constant magnitude} if $a,b$ are constant.
In the \emph{unit-speed class}, $|\gamma'|=1$ and $t$ is profile arc length.
At a coordinate axis we use the smooth components $z_i$ instead of magnitude--phase variables.

We encode the source warps by
$
          F=\ell_1^{k_1}\ell_2^{k_2},
       \,
          T=\frac{d_2}{\ell_2^2}-\frac{d_1}{\ell_1^2}.
$
We also set $\lambda_i=d_i/\ell_i^2$, so that $T=\lambda_2-\lambda_1$.
Conversely, in Lemma \ref{gsep} we show that $(F,T)$ determines a unique smooth positive warp pair.
For $0<a<1$, the level $|z_1|=a$, $|z_2|=b$ is a Hopf torus, while either axis is a coordinate circle.
The phases determine the motion on this level, so a magnitude limit alone does not determine the accumulation set.

\subsection{Main results}
We first identify the common graph geometry in Theorem A.
We then vary the source metric through inverse reconstruction and phase matching in Theorems B and C.
In Theorem D we keep the full source metric fixed and select closed Neumann orbits.

\medskip
\noindent\textbf{Theorem A (prescribed-warp graph geometry).}
\emph{Set $V(t,z)=\tfrac12(\lambda_1|z_1|^2+\lambda_2|z_2|^2)$.
For $(\xi,v)\in T_tI\times T_z\Sph^3$ with $\xi>0$, define
\begin{equation*}
      \mathcal K(t,z;\xi,v)
        =\frac{F(t)}{2\xi}\bigl(|v|^2+2V(t,z)\xi^2\bigr).
\end{equation*}
Then $\mathcal K$ is a strongly convex conic Finsler metric on $I\times\Sph^3$.
The product $G_\gamma$ is harmonic if and only if the oriented graph $t\mapsto(t,\gamma(t))$ is an unparametrized $\mathcal K$-geodesic (Theorem \ref{hkrop}).}

The cone $\xi>0$ selects positive time orientation.
The warps determine $\mathcal K$, while an initial position and direction determine an individual profile.

For the unit-speed inverse problem, interchange the factors if necessary and normalize the two conserved phase momenta by
$
       \bigl(Fa^2s_1',Fb^2s_2'\bigr)
          =\sigma\sqrt\varepsilon\,(1,c),
       \, \varepsilon>0,
       \, c\in\R,
       \, \sigma\in\{-1,1\}.
$
Thus harmonicity reduces to one scalar magnitude equation.

\medskip
\noindent\textbf{Theorem B (unit-speed reconstruction).}
\begin{enumerate}
\item \emph{If $a:I\to(0,1)$ is smooth and $R=1-a^2-(a')^2>0$, then it determines the other magnitude, both source warps, and the phases up to constant rotations.  With $Q=1+(c^2-1)a^2$, the reconstruction is governed by}
$
       F^2=\frac{\varepsilon Q}{a^2R},
       \,
       T=\frac1{a^2Q}-\frac{a+a''}{aR}.
$
\emph{Turns and constant intervals are allowed.}
\item \emph{If one source warp is prescribed on $\R$ and both phase momenta are nonzero, every admissible initial position and direction gives a unique global reconstruction.  The source is complete when the input factors are complete.}
\item \emph{If a prescribed positive volume factor is strictly monotone and has positive infimum, a global reconstruction exists exactly when}
$
       \sqrt\varepsilon(1+|c|)\le \inf_{\R}F.
$
\emph{Periodic volumes and finite or infinite volume-factor limits are treated in Section \ref{sevol}.}
\end{enumerate}

Theorem B allows us to reconstruct the source geometry from partial data.
For periodic data, however, scalar return is not enough: we must also close the two phase increments.

\medskip
\noindent\textbf{Theorem C (periodic-warp closure).}
\emph{Let one positive periodic source warp be fixed and assume that the input manifolds are closed.}
\begin{enumerate}
\item \emph{There are local families with periodic magnitudes and periodic reconstructed source metrics in which profiles closing on finite covers are dense.}
\item \emph{There is a closed doubly spiral sequence with uniformly positive source warps such that one target magnitude tends to zero and the least closing multiplier tends to infinity.}
\item \emph{Selected nondegenerate closed members persist, with the same phase return, under sufficiently small positive periodic perturbations of the prescribed warp.}
\end{enumerate}
\emph{These statements are proved in Section \ref{seper}.}

Here the second warp, and hence the source metric, still varies.
We change viewpoint in the next result: we fix the full source metric and find closed profiles on it.

\medskip
\noindent\textbf{Theorem D (fixed-metric closure and unbounded degree).}
\emph{Let $F>0$ be smooth and $P$-periodic, and set $I_F=\int_0^P F^{-1}\,dt$.}
\begin{enumerate}
\item \emph{For every $\kappa\in\R$, the data $F$ and $T=\kappa/F^2$ determine a fixed positive periodic source metric.  Its harmonic profiles are time changes of the trajectories of an autonomous Neumann system.}
\item \emph{For coprime nonzero integers $r_1,r_2$ with $|r_1|\ne|r_2|$, choose}
$
       \kappa=\frac{4\pi^2(r_1^2-r_2^2)}{I_F^2}.
$
\emph{The resulting single metric carries a one-parameter family of pairwise noncongruent closed profiles on distinct Hopf tori.  Their oriented graphs are pairwise disjoint closed $\mathcal K$-geodesics of equal length.}
\item \emph{For the real identity inputs on $\Sph^p$ and $\Sph^q$, with $p,q\ge2$ even, every fixed metric in this class supports harmonic maps}
$
       (\R/P\Z)\times\Sph^p\times\Sph^q
          \longrightarrow \Sph^{p+q+1}
$
\emph{of degree $4m$ for all $m\ge1$.  Their closed graph geodesics are not iterates of shorter ones, and their lengths are unbounded.}
\end{enumerate}

For some matched constant source warps in this family, the corresponding maps are eigenmaps.

We organize the paper as follows.
In Sections \ref{sepre} and \ref{segeo}, we derive the profile equation and describe its graph geometry.
We develop inverse reconstruction and periodic closure in Sections \ref{seglb} and \ref{seper}.
For the fixed-metric problem, Section \ref{seneu} focuses on the Neumann subclass.
Finally, in Section \ref{sevol} we study the prescribed volume factor.

\section{The common profile equation}\label{sepre}

Both constructions use the same one-dimensional reduction.
We derive it together with the conserved phase momenta and a formulation that remains regular at the coordinate axes.

Minimal isometric immersions have $d=k$ by Takahashi's theorem \cite{T66}.
Further examples include degree-$p$ spherical-harmonic eigenmaps, with $d=p(p+k-1)$, and circle windings $e^{i\theta}\mapsto e^{im\theta}$, with $d=m^2$ \cite{ER93}.

For constant positive $a,b,\ell_1,\ell_2$ with $a^2+b^2=1$, the map $(af_1,bf_2)$ from $(M_1\times M_2,\ell_1^2g_1+\ell_2^2g_2)$ is harmonic precisely when $d_1/\ell_1^2=d_2/\ell_2^2$.
Taking $\ell_1=a$ and $\ell_2=b$, we obtain $a^2=d_1/(d_1+d_2)$ and $b^2=d_2/(d_1+d_2)$.
When the inputs are minimal isometric immersions, so that $d_i=k_i$, these are the classical minimal-product coefficients.

\subsection{Component equations and phase momenta}

Let $g=g_\gamma$, let a prime denote differentiation in $t$, and set $q=|\gamma'|^2=(a')^2+(b')^2+a^2(s_1')^2+b^2(s_2')^2$.
Also set
\begin{equation}
     K=F'/F,
     \qquad E=q+a^2d_1/\ell_1^2+b^2d_2/\ell_2^2=|dG_\gamma|^2.
\end{equation}
The warped Laplacian is $\Delta_g=\partial_t^2+K\partial_t+\ell_1^{-2}\Delta_{g_1}+\ell_2^{-2}\Delta_{g_2}$.

In the principal region, we decompose $\Delta G_\gamma=((A_1+iA_2)ae^{is_1}f_1,(B_1+iB_2)be^{is_2}f_2)$ and obtain
\begin{equation*}
     \begin{aligned}
         A_1&=-\frac{d_1}{\ell_1^2}+\frac{a''}a-(s_1')^2+K\frac{a'}a,
     & A_2&=s_1''+\left(2\frac{a'}a+K\right)s_1',\\
     B_1&=-\frac{d_2}{\ell_2^2}+\frac{b''}b-(s_2')^2+K\frac{b'}b,
       & B_2&=s_2''+\left(2\frac{b'}b+K\right)s_2'.
     \end{aligned}
\end{equation*}
Comparing the real and imaginary parts, we find that harmonicity is equivalent to
\begin{equation}\label{phrm}
        A_1=B_1=-E,\qquad A_2=B_2=0.
\end{equation}
Integrating the imaginary equations, we obtain
\begin{equation}\label{ppha}
       (Fa^2s_1')'=0,\qquad (Fb^2s_2')'=0.
\end{equation}
Hence the two constants on each connected principal-region interval are
\begin{equation}\label{pfirst}
       C_1=Fa^2s_1',\qquad C_2=Fb^2s_2'.
\end{equation}
Their coordinate-free expressions $C_i=F\operatorname{Im}(\overline z_i z_i')$ remain defined at a coordinate-axis crossing.
Differentiating $a^2+b^2=1$ and using the equations also yields
\begin{equation}\label{psphe}
         a^2A_1+b^2B_1=-E,
\end{equation}
the real equations reduce to $A_1=B_1$.

Conversely, suppose that harmonic sphere maps $f_i$ produce a harmonic product with $a,b>0$.
By separating the $x$- and $y$-dependence, we find that each $|df_i|^2$ must be constant, so both inputs are eigenmaps.

\begin{remark}[Eigenmap criterion]
With the notation above, every harmonic product satisfies
\begin{equation*}
       \Delta G_\gamma=-E(t)G_\gamma,
       \qquad
       E(t)=|\gamma'|^2+\lambda_1|z_1|^2+\lambda_2|z_2|^2.
\end{equation*}
The product $G_\gamma$ is an eigenmap precisely when $E\equiv\mu$ for some constant $\mu$.
Equivalently, its profile components satisfy the weighted Sturm--Liouville equations
\begin{equation*}
          -\frac1F(Fz_i')'+\lambda_i z_i=\mu z_i,
          \qquad i=1,2.
\end{equation*}
\end{remark}

\subsection{Variational form of the profile equation}

The polar equations are useful for reconstruction but become singular when a component vanishes.
We therefore pass to a variational equation on $\Sph^3$, which remains smooth and leads directly to the graph geometry.

For fixed warps let $\Lambda=\operatorname{diag}(\lambda_1,\lambda_2)$.
On a compact time interval $J$ the reduced energy is
\begin{equation}\label{pener}
        \mathcal E_J(\gamma)
       =\frac12\int_JF\bigl(|\gamma'|^2+\langle\Lambda\gamma,\gamma\rangle\bigr)\,dt,
\end{equation}
where the inner product is real Euclidean.
When $M_1$ and $M_2$ have finite volume, this is the energy of $G_\gamma$ divided by $\operatorname{vol}(M_1)\operatorname{vol}(M_2)$.

By varying $\gamma$ in $\Sph^3$ with fixed endpoints, we obtain
\begin{equation}\label{peule}
     \gamma''+K\gamma'-\Lambda\gamma
          +\left(|\gamma'|^2+\langle\Lambda\gamma,\gamma\rangle\right)\gamma=0.
\end{equation}
We recover \eqref{phrm} from this equation, while \eqref{ppha} follows from its phase symmetries.
Unlike the polar system, it remains regular when a profile component vanishes.

\section{Graph geometry for prescribed source warps}\label{segeo}

For prescribed source warps, we interpret the reduced equation as a graph-geodesic equation.
Throughout this section $F$ and $T$ are arbitrary.
The autonomous fixed-metric sector of Section \ref{seneu} arises from the additional condition $F^2T=\kappa$ for a constant $\kappa$.
With $V$ as in Theorem A, the reduced profile Lagrangian is
\begin{equation*}
          L(t,z,u)=\frac{F(t)}2\bigl(|u|^2+2V(t,z)\bigr).
\end{equation*}
Following the positive homogenization of a time-dependent Lagrangian \cite[Section 4]{Me16}, we consider the cone $\mathscr C=\{(t,z;\xi,v)\in T(I\times\Sph^3):\xi>0\}$ and define
\begin{equation}\label{hkmet}
      \mathcal K(t,z;\xi,v)
      =\frac{F(t)}{2\xi}\bigl(|v|^2+2V(t,z)\xi^2\bigr).
\end{equation}

\begin{theorem}[Graph-geodesic formulation]\label{hkrop}
The function $\mathcal K$ is a smooth, positively $1$-homogeneous, strongly convex conic Finsler metric on $\mathscr C$.
The map $G_\gamma$ is harmonic if and only if the oriented graph $t\mapsto(t,\gamma(t))$ is an unparametrized $\mathcal K$-geodesic.
This formulation is regular at magnitude turns and vanishing profile components.
\end{theorem}
\begin{proof}
Since $\lambda_i>0$, we have $V>0$.
Consequently,
\[
   \mathcal K=\frac{\alpha^2}{\beta},
   \qquad
   \alpha^2=F(g_{\Sph^3}+2Vdt^2),
   \qquad
   \beta=2dt,
\]
is a strongly convex conic Kropina metric on $\mathscr C$
by \cite[Corollary 4.12]{JS14}.
We parametrize every admissible curve by $t$.
In the gauge $\dot t=1$, its $\mathcal K$-length becomes \eqref{pener},
and the $\Sph^3$ Euler--Lagrange equation is \eqref{peule}.
The $1$-homogeneity of $\mathcal K$ supplies the remaining equation and proves the equivalence.
Equation \eqref{peule} also shows directly that the formulation remains regular at turns and coordinate axes.
\end{proof}

\begin{remark}[Noether integrals]
The componentwise action of $\mathbb T^2=U(1)^2$ on $\Sph^3\subset\C^2$ preserves $\mathcal K$.
For its infinitesimal generators $Z_1(z)=(iz_1,0)$ and $Z_2(z)=(0,iz_2)$, the Finsler Noether integrals along a graph geodesic are
\begin{equation*}
       \left.d_y\mathcal K(0,Z_i)\right|_{(t,\gamma;1,\gamma')}
          =F\langle\gamma',Z_i(\gamma)\rangle
          =F\operatorname{Im}(\overline z_i z_i')
          =C_i .
\end{equation*}
We thus recover exactly the constants in \eqref{pfirst}.
\end{remark}

\begin{remark}\label{hperi}
If both warps are $P$-periodic, then $\mathcal K$ descends to $(\mathbb R/P\mathbb Z)\times\Sph^3$, and its positively oriented closed geodesics with integer time winding $m\ge1$ correspond to $mP$-periodic harmonic profile curves.
\end{remark}

\subsection{Momentum constraints, axis crossing, and joins}

We now use the conserved momenta to constrain source collapse, axis crossing, and the gluing of profile pieces.

\begin{remark}[Momentum and source collapse]
The momenta $C_i=F\operatorname{Im}(\overline z_i z_i')$ satisfy
$$
     |C_1|+|C_2|\le F|\gamma'|,
         \qquad
          F|\gamma'|^2\ge (|C_1|+|C_2|)^2/F.
$$
Thus $C_1=C_2=0$ if $F\to0$ along an end while $|\gamma'|$ remains bounded.
The same conclusion holds on any interval $U\subset I$ for which $\int_Udt/F=\infty$ and $\int_UF|\gamma'|^2dt<\infty$.
Thus nonzero momentum obstructs source collapse.
\end{remark}

\begin{proposition}[Axis crossing and phase rigidity]
Let $\gamma=(z_1,z_2)$ solve \eqref{peule} on a connected interval on which $F>0$.
Then $C_i=F\operatorname{Im}(\overline z_i z_i')$, $i=1,2$, are constant.
If $z_i$ vanishes, then $C_i=0$.
Conversely, if $C_i=0$, then either $z_i\equiv0$ or a constant phase rotation makes $z_i$ real-valued.
If both coordinates have a zero, possibly at different times, a fixed $U(1)^2$ rotation moves $\gamma$ into the real great circle.
If $C_i\ne0$, the $i$th target block never vanishes.
Hence $G_\gamma$ omits every point of the coordinate subsphere $\{(w_1,w_2)\in\Sph^{2n_1+2n_2+3}:w_i=0\}$ and every such closed product is null-homotopic in the ambient sphere.
\end{proposition}

\begin{proof}
The $i$th component of \eqref{peule} has the form
\[
        z_i''+\frac{F'}Fz_i'+r_i(t)z_i=0,
        \qquad
        r_i(t)=|\gamma'|^2+\langle\Lambda\gamma,\gamma\rangle-\lambda_i(t),
\]
and $r_i$ is real.
Multiplying by $\overline z_i$ and taking imaginary parts, we obtain conservation of $C_i$.
A zero of $z_i$ forces $C_i=0$.
If $C_i=0$ and $z_i\not\equiv0$, a constant phase rotation makes the Cauchy data real at a point where $z_i\ne0$.
Uniqueness keeps $z_i$ real throughout.
If $C_i\ne0$, then $z_i(t)\ne0$ everywhere.
Choose any point $p$ of the omitted coordinate subsphere.
The image lies in $\Sph^{2n_1+2n_2+3}\setminus\{p\}$, which is contractible.
\end{proof}

The conserved momenta prevent a smooth join from changing which coordinates rotate.
In particular, a real Smith-type join cannot acquire a rotating phase at an axis crossing.
Joining constant and nonconstant magnitudes is possible in the variable-metric reconstruction (Remark \ref{massy}), but not by leaving a fixed Neumann relative equilibrium with the same Cauchy data.

\begin{remark}[Spiral minimal products]\label{hlzrc}
Suppose the inputs are horizontal isometric minimal immersions and the profile has unit speed.
Here horizontality means $\langle if_i,df_i(X)\rangle_{\mathbb R}=0$.
Impose the induced metric $\ell_1=a=\cos s$, $\ell_2=b=\sin s$.
Since $d_i=k_i$, we find $F=W(s)$ and $abT=(\log W)_s$, where $W(s)=\cos^{k_1}s\sin^{k_2}s$.
Under these hypotheses, equation \eqref{peule} agrees with the weighted-geodesic equation in the previous work of Li--Zhang on spiral minimal products \cite{LZ26}.
This identifies the shared profile geometry on $\Sph^3$, while the free data and geometric constraints of the harmonic and minimal constructions remain different.
Combined with unit speed, horizontality yields $G_\gamma^*g_{\rm round}=g_\gamma$.
Consequently, every harmonic product in this induced-metric sector is a minimal isometric immersion and an eigenmap with eigenvalue $1+k_1+k_2$.
\end{remark}

\section{Variable-metric inverse reconstruction}\label{seglb}

Within the unit-speed class, we consider two inverse problems.
We may prescribe an admissible magnitude and reconstruct both warps, or prescribe the first warp and let the evolution determine the second.
Here and below, \emph{entire} means defined for all $t\in\R$.

\subsection{Identities and separation of the warps}

Throughout this section the profile curve has unit speed, and $b=\sqrt{1-a^2}$ in the principal region.
For a harmonic product on a connected interval with $a,b>0$ and $s_1'\ne0$, we use the momentum normalization and the functions $Q,R$ introduced in Theorem B.
Thus $C_1=\sigma\sqrt\varepsilon$ and $C_2=cC_1$ in \eqref{pfirst}.
Combining unit speed with the inner product of \eqref{peule} and $\gamma'$, we obtain
\begin{equation}\label{gsca}
     R=\frac{\varepsilon Q}{F^2a^2}>0,
     \qquad
     (a')^2=1-a^2-\frac{\varepsilon Q}{F^2a^2},
\end{equation}
\begin{equation}\label{gcmp}
        \frac{F'}F+Taa'=0.
\end{equation}
For fixed $(c,\varepsilon,\sigma)$, we determine $F$ from \eqref{gsca} and $T$ from $A_1=-E$.
The following lemma then determines the warps, while \eqref{ppha} determines the phases.
These identities remain valid at turns.

\begin{lemma}[Separation of the warps]\label{gsep}
For every $F>0$ and $T\in\mathbb R$ there is a unique positive pair $(\ell_1,\ell_2)$ satisfying
$$
       \ell_1^{k_1}\ell_2^{k_2}=F,
          \qquad
       \frac{d_2}{\ell_2^2}-\frac{d_1}{\ell_1^2}=T.
$$
The pair depends smoothly on $(F,T)$.
\end{lemma}

\begin{proof}
Set $x=\ell_1^2$.
The second equation becomes $d_2F^{-2/k_2}x^{k_1/k_2}-d_1/x=T$, whose left-hand side is strictly increasing from $-\infty$ to $+\infty$.
Existence and uniqueness follow, and the implicit-function theorem provides smooth dependence.
\end{proof}

\subsection{Prescribing the magnitude}

\begin{theorem}[Prescribed-magnitude reconstruction]\label{grec}
Fix $c\in\mathbb R$, $\varepsilon>0$, and $\sigma\in\{-1,1\}$.
Let $I$ be an interval, and let $a\in C^\infty(I,(0,1))$ be admissible in the sense of Theorem B.
With the functions $Q$ and $R$ fixed there, define
\begin{equation}\label{greg}
         \begin{gathered}
        b=\sqrt{1-a^2},
         \qquad F=\sqrt{\frac{\varepsilon Q}{a^2R}},
        \\
         T=\frac1{a^2Q}-\frac{a+a''}{aR}.
         \end{gathered}
\end{equation}
Let $(\ell_1,\ell_2)$ be the pair in Lemma \ref{gsep}, and choose phases with
\begin{equation}\label{gpha}
       s_1'=\frac{\sigma\sqrt\varepsilon}{Fa^2},
      \qquad s_2'=\frac{\sigma c\sqrt\varepsilon}{Fb^2}.
\end{equation}
Then $G=(ae^{is_1}f_1,be^{is_2}f_2)$ is harmonic for $g=dt^2+\ell_1^2g_1+\ell_2^2g_2$, and its profile curve has unit speed.
For fixed $(a,c,\varepsilon,\sigma)$ the reconstruction is unique up to additive constants in the phases.
Conversely, every unit-speed harmonic spiral product with $a,b>0$ and $s_1'\ne0$ arises this way on its connected interval.
\end{theorem}

\begin{proof}
By differentiating the first formula in \eqref{greg}, we obtain
\begin{equation}\label{gder}
     \frac{F'}F
      =a'\left(-\frac1{aQ}+\frac{a+a''}{R}\right)
        =-Taa',
\end{equation}
\begin{equation}\label{grad}
          a''=\frac{\varepsilon}{F^2a^3}-a-RTa.
\end{equation}
The phases satisfy \eqref{ppha}, and direct substitution yields
$$
     |\gamma'|^2
          =\frac{(a')^2+R}{b^2}=1,
         \qquad
      A_1+E=\frac{a''}{a}-\frac{\varepsilon}{F^2a^4}+1+RT=0.
$$
We then use \eqref{psphe} to obtain $B_1=-E$, which proves \eqref{phrm} also at turns.
Conversely, \eqref{gsca} determines $F$.
Combined with the first real equation, \eqref{gcmp} yields \eqref{grad} and hence $T$.
Finally, Lemma \ref{gsep} and integration of \eqref{gpha} establish uniqueness.
\end{proof}

\begin{remark}
Since $R\le1-a^2$, Cauchy--Schwarz provides the necessary bound
\begin{equation}\label{glow}
         F^2\ge\varepsilon\left(\frac1{a^2}+\frac{c^2}{1-a^2}\right)
     \ge\varepsilon(1+|c|)^2.
\end{equation}
For $c\ne0$, equality in the final bound holds exactly when $a'=0$ and $a^2=1/(1+|c|)$.
For $c=0$, $F^2\ge\varepsilon/a^2>\varepsilon$.
\end{remark}

If both phase momenta vanish, the unit-speed condition forces a meridian $a=\cos s$, $b=\sin s$, $s'=\pm1$, which cannot remain in the principal region on $\mathbb R$.
Hence Theorem \ref{grec} covers every entire unit-speed principal-region product, after interchanging the factors if necessary.

\begin{remark}[Joining magnitude arcs]\label{massy}
For fixed $(c,\varepsilon,\sigma)$, we can extend any finite ordered collection of compact admissible magnitude arcs to an entire admissible magnitude function with constant tails.
Indeed, with $a=\cos s$, admissibility is $0<s<\pi/2$ and $|s'|<1$, conditions preserved by sufficiently slow $C^\infty$-flat transitions.
Theorem \ref{grec} then reconstructs both warps and the phases.
On a constant tail the warps are constant and the phase-slope ratio is $s_2'/s_1'=ca^2/(1-a^2)$.
The orbit on its carrier torus is closed when this ratio is rational and dense when it is irrational.
\end{remark}

\begin{figure}[!htbp]
\centering
\includegraphics[width=.9\linewidth]{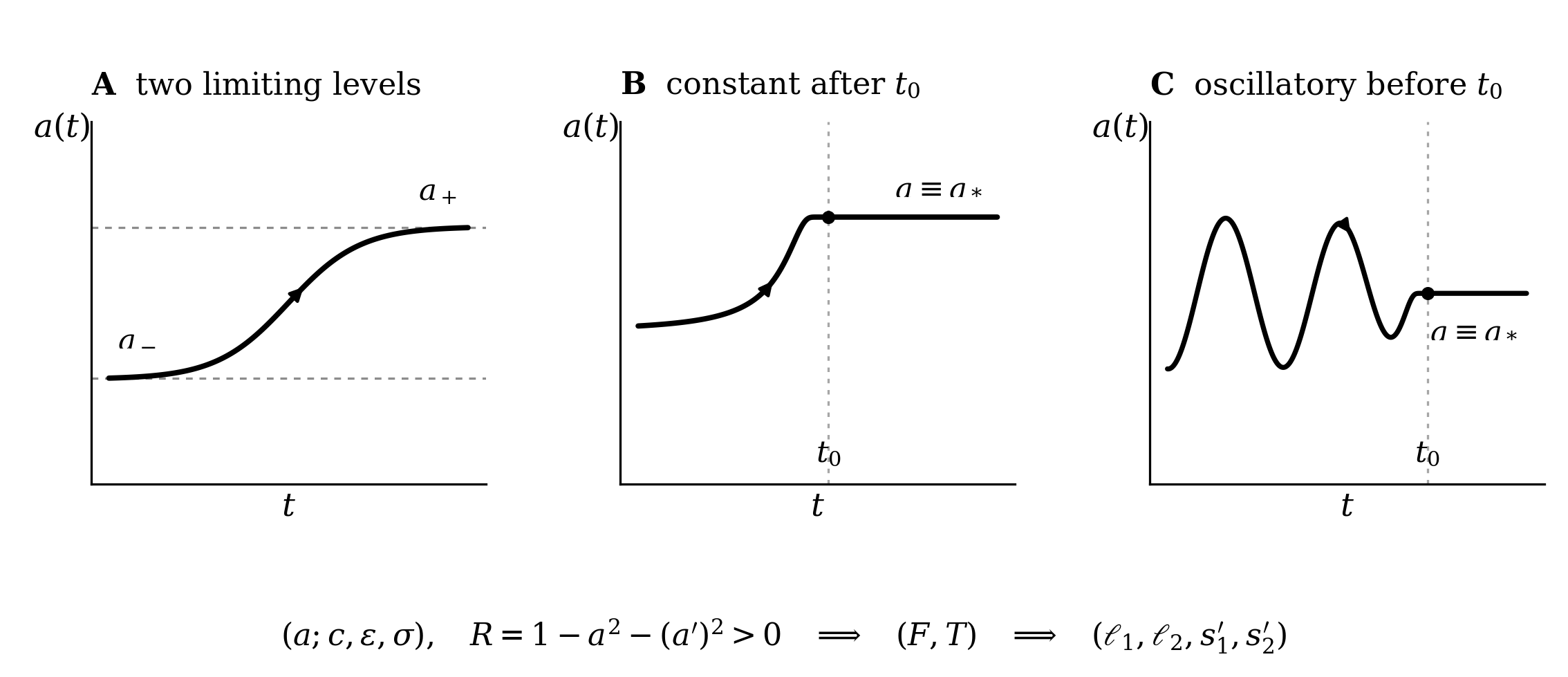}
\caption{Some admissible magnitude profiles used in the inverse construction
}
\label{fprofile}
\end{figure}
\FloatBarrier

Define $\chi(x)=e^{-1/x^2}$ for $x>0$ and $\chi(x)=0$ for $x\le0$, and consider
\begin{equation*}
      \text{(A)}\quad a_{\rm as}(t)=\frac3{10}+\frac7{20(1+e^{-t})},
      \quad \text{(B)}\quad a_{\rm flat}(t)=\frac34-\frac14\chi(-t),
      \quad \text{(C)}\quad a_{\rm osc}(t)=\frac12+\frac14\chi(-t)\sin t.
\end{equation*}
Here $a_{\rm flat}$ is nonconstant for $t<0$ and equals $3/4$ for $t\ge0$, with all derivatives matching at $t=0$.
The function $a_{\rm osc}$ describes an oscillatory transition of the same type.
These functions satisfy $1/4<a\le3/4$, $|a'|<1/2$, and hence $R>3/16$.
These three choices produce the patterns in Figure \ref{fprofile}; applying Theorem \ref{grec}, we reconstruct the corresponding warps and phases.

\subsection{Prescribing one warp globally}

Set $\mathcal D=\{(a,p):0<a<1,\ 1-a^2-p^2>0\}$.

Once we prescribe $\ell_1$, the second warp $\ell_2$, and hence the source metric, remains undetermined.
We therefore fix an initial state $(a_0,p_0)\in\mathcal D$, which specifies the magnitude and its initial direction, and let the evolution select $\ell_2$.
In particular, $F(t_0)^2=\varepsilon Q(a_0)/[a_0^2(1-a_0^2-p_0^2)]$, so varying $p_0$ generally produces different source metrics.

\begin{theorem}[One prescribed warping function]\label{gwar}
Fix a positive $\ell_1\in C^\infty(\mathbb R)$, $c\ne0$, $\varepsilon>0$, $\sigma\in\{-1,1\}$, $t_0\in\mathbb R$, and $(a_0,p_0)\in\mathcal D$.
There is a unique smooth magnitude function $a$ on $\mathbb R$ with $a(t_0)=a_0$, $a'(t_0)=p_0$, and $(a,a')\in\mathcal D$, whose reconstruction in Theorem \ref{grec} has the prescribed first warp.
The second warp is positive and smooth, and the phases are unique up to additive constants.
The resulting source metric is complete if both factors are complete.
\end{theorem}

\begin{proof}
On $\mathbb R\times\mathcal D$, let
$
 R=1-a^2-p^2,
 \,
  Q=1+(c^2-1)a^2,
 \,
  F=\sqrt{\frac{\varepsilon Q}{a^2R}},
  \,
 L_2=\left(\frac{F}{\ell_1^{k_1}}\right)^{1/k_2},
 \,
  T=\frac{d_2}{L_2^2}-\frac{d_1}{\ell_1^2}.
$
The right-hand side is smooth on $\mathbb R\times\mathcal D$, so the following initial-value problem has a unique maximal solution:
\begin{equation}\label{gvec}
          a'=p,
       \qquad p'=\frac{\varepsilon}{F^2a^3}-a-RTa,
          \qquad (a,p)(t_0)=(a_0,p_0).
\end{equation}
Differentiating the definitions along a local solution, we find $F'/F=-Tap$.
Consequently, $W=F^{2/k_2}$ satisfies
\begin{equation}\label{gvolw}
       W'=\frac{2d_1aa'}{k_2\ell_1^2}W
      -\frac{2d_2aa'}{k_2}\ell_1^{2k_1/k_2}.
\end{equation}
Let $I_{\max}$ be the maximal interval, and fix a compact interval $J$ containing $t_0$.
On $I_{\max}\cap J$ we have $|aa'|\le1/2$, while $\ell_1$ and $\ell_1^{-1}$ are bounded.
It follows that $|W'|\le A_JW+B_J$, and Gronwall's inequality provides a bound $F\le M_J$ on $I_{\max}\cap J$.
Combined with \eqref{glow}, this gives
\begin{equation*}
       a^2\ge\frac{\varepsilon}{M_J^2},
       \qquad
       1-a^2\ge\frac{\varepsilon c^2}{M_J^2},
       \qquad
       R=\frac{\varepsilon Q}{a^2F^2}
          \ge\frac{\varepsilon\min\{1,c^2\}}{M_J^2}>0.
\end{equation*}
Thus $(a,a')$ remains in a compact subset of $\mathcal D$ on $I_{\max}\cap J$.
The continuation criterion excludes any finite endpoint, and therefore $I_{\max}=\mathbb R$.
Since every reconstruction with first warp $\ell_1$ must satisfy \eqref{gvec}, we also obtain uniqueness.
We then apply Lemma \ref{gsep}, \eqref{grad}, and Theorem \ref{grec} to conclude harmonicity and phase uniqueness.
Finally, consider a finite-length source curve.
Its $t$-coordinate lies in a compact interval on which both warps are bounded below, and its factor projections have finite length.
When the factors are complete, these projections converge, which proves completeness of the source metric.
\end{proof}

\begin{remark}
The condition $c\ne0$ is needed for the assertion about every initial state.
For $c=0$, use the signed amplitude $b$ and set $S=1-b^2-(b')^2$, $F=\sqrt{\varepsilon/S}$, and $\ell_2=(F/\ell_1^{k_1})^{1/k_2}$.
The equation $b''=-b+TbS$ is smooth at $b(t_*)=0$, $b'(t_*)\in(0,1)$, and its solution crosses the axis.
On the adjacent positive side, $a=\sqrt{1-b^2}$ has $R=b^2S/(1-b^2)>0$.
Hence a principal-region solution can reach its boundary in finite time.
\end{remark}

\section{Periodic closure with one prescribed warp}\label{seper}

The preceding section provides a flexible inverse construction, but it does not by itself close the profile.
We now prescribe a positive $P$-periodic first warp $\ell_1=L$ and determine when the reconstructed data produce a closed profile.
Periodicity of the magnitude and source metric is only the first step: the two phase increments must also have rational return.
When $M_1$ and $M_2$ are closed, such a return allows us to obtain a harmonic map from the closed manifold $S^1_{mP}\times M_1\times M_2$.
Here $S^1_{mP}=\R/(mP\Z)$.

\begin{proposition}[Phase-return criterion]\label{ccrit}
For an admissible smooth $P$-periodic magnitude with fixed $(c,\varepsilon,\sigma)$, the reconstructed warps and phase velocities are $P$-periodic.
Set $\Delta_i=\int_0^P s_i'\,dt$ for $i=1,2$.
The profile is $mP$-periodic exactly when
$
        m(\Delta_1,\Delta_2)\in(2\pi\mathbb Z)^2.
$
Hence, when $M_1$ and $M_2$ are closed, rational phase return produces a harmonic product on a finite cover of the base circle.
\end{proposition}
\begin{proof}
The reconstruction formulas are $P$-periodic, and $s_i(t+mP)-s_i(t)=m\Delta_i$.
\end{proof}

\subsection{Closing with the first warp fixed}

In Appendix \ref{apcl} we construct small $P$-periodic magnitude seeds for every $L$.
Starting from these seeds, we prove that the phase-increment map is nondegenerate.

\begin{theorem}[Closing with the first warp fixed]
For every $c_*\in\R$ and all sufficiently large $n$, there are $P$-periodic reconstructions with $\ell_1=L$, smoothly parametrized by $(\delta,c)$ near $(\delta_n,c_*)$, where $\delta_n\downarrow0$ and $\varepsilon=\delta^{2(k_2+1)}$.
The phase-increment map has rank two, so rational returns are dense.
On the slice $c=0$, its first component has nonzero $\delta$-derivative, and rational returns are dense there as well.
When $M_1$ and $M_2$ are closed, the returning members give harmonic maps from closed source manifolds on finite $P$-covers.
\end{theorem}

\begin{proof}
Use Theorem \ref{pall} with $\sigma=1$ and $a=\delta y$.
Set $R_\delta=1-\delta^2(y^2+y'^2)$ and $Q_\delta=1+(c^2-1)\delta^2y^2$.
The phase increments are
\[
       \Theta_1=\frac1\delta\int_0^P\frac{\sqrt{R_\delta}}{y\sqrt{Q_\delta}}\,dt,
       \qquad
       \Theta_2=c\delta\int_0^P\frac{y\sqrt{R_\delta}}{(1-\delta^2y^2)\sqrt{Q_\delta}}\,dt.
\]
At the base points, set $J_1=\int_0^P y_0^{-1}\,dt>0$ and $J_2=\int_0^P y_0\,dt>0$.
Using the smooth dependence established in Appendix \ref{apcl}, differentiating under the integral, and applying \eqref{pfbd}, we obtain
\[
 \begin{aligned}
 \partial_\delta\Theta_1&=-J_1\delta^{-2}+O(\delta^{-1}),
 &\partial_c\Theta_1&=O(1),\\
 \partial_\delta\Theta_2&=cJ_2+O(\delta),
 &\partial_c\Theta_2&=\delta J_2+O(\delta^2).
 \end{aligned}
\]
Consequently,
\[
       \det D_{(\delta,c)}(\Theta_1,\Theta_2)
          =-\frac{J_1J_2}{\delta}+O(1)\ne0,
       \qquad
       \partial_\delta\Theta_1(\delta,0)
          =-\frac{J_1}{\delta^2}+O(\delta^{-1})\ne0.
\]
The inverse-function theorem, combined with Proposition \ref{ccrit}, now proves the stated density and closure.
\end{proof}

If $L$ is nonconstant, none of these magnitudes is constant: otherwise $F,T$, and hence $L$ by Lemma \ref{gsep}, would be constant.

\subsection{A periodic branch with uniformly positive warps}

In the preceding branch, $F\sim\delta^{k_2}$ and the second source warp collapses.
We now choose a different scaling, for which the momentum ratio grows while $\varepsilon c^2$ has a nonzero limit.
As a result, both source warps remain uniformly positive, and the limiting magnitude equation becomes a periodic Pinney equation.

\begin{theorem}[Periodic branch with uniformly positive warps]\label{pnonc}
For every positive smooth $P$-periodic $L$, there are $\mu_*>0$, $\delta_0>0$, and a smooth two-parameter family of doubly spiral reconstructions near $(\mu_*,0)$, with $0<\delta<\delta_0$ and $\ell_1=L$.
Their magnitudes and source warps are $P$-periodic.
As $\delta\to0$, one has $a\to0$ in $C^\infty$, while both warps remain uniformly bounded above and away from zero.
The phase-increment map $(\mu,\delta)\mapsto(D_1,D_2)$ has rank two.
\end{theorem}

\begin{proof}
Set $k=k_2$, $\rho=d_1/L^2$ and $w=L^{2k_1/k}$.
By Lemma \ref{pnell}, we choose $\mu_*>0$ with $\mu_*w>\rho$ and a positive periodic solution $x_{\mu_*}$ of
\begin{equation}\label{pnlim}
      x_\mu''+(1+\mu w-\rho)x_\mu=x_\mu^{-3}
\end{equation}
whose periodic linearization is invertible.
Set
\[
         c=\delta^{-1},
      \qquad \varepsilon=\delta^2(d_2/\mu)^k,
        \qquad a=\sqrt\delta\,x,
\]
with $q_\delta=x^2+\delta-\delta^2x^2$, $r_\delta=1-\delta(x^2+x'^2)$.
Then $Q=\delta^{-1}q_\delta$, $R=r_\delta$, and \eqref{gvec} becomes
\begin{equation}\label{pnex}
      x''=\frac{r_\delta}{xq_\delta}-x
          -r_\delta x\left[
       \mu w\left(\frac{x^2r_\delta}{q_\delta}\right)^{1/k}-\rho
     \right].
\end{equation}
For $0<\alpha<1$, its residual near $(\mu_*,0,x_{\mu_*})$ is smooth as a map $\mathbb R^2\times C^{2,\alpha}_{\rm per}\to C^{0,\alpha}_{\rm per}$, also for small negative auxiliary $\delta$, since $x,q_\delta,r_\delta$ remain positive.
Applying the implicit-function theorem and then bootstrapping, we obtain $x_{\mu,\delta}$ with smooth dependence in every $C^j$ norm near a rectangle $\overline J\times[0,\delta_0]$, where $\mu_*\in J\Subset(0,\infty)$.
For $\delta>0$ the reconstruction is admissible, with
\[
      F=(d_2/\mu)^{k/2}
     \sqrt{\frac{q_\delta}{x^2r_\delta}},
          \qquad
          F\longrightarrow(d_2/\mu)^{k/2},
          \qquad
          \ell_2\longrightarrow\sqrt{d_2/\mu}\,L^{-k_1/k}
\]
in $C^\infty$, uniformly for $\mu\in\overline J$; these limits provide the required warp bounds.

For $\sigma=1$ the phase increments extend smoothly to $\delta=0$:
\[
     D_1=\int_0^P\frac{\sqrt{r_\delta}}{x\sqrt{q_\delta}}\,dt,
          \qquad
       D_2=\int_0^P
     \frac{x\sqrt{r_\delta}}
         {(1-\delta x^2)\sqrt{q_\delta}}\,dt.
\]
Here $D_1(\mu,0)=\int_0^P x_\mu^{-2}\,dt$ and $D_2(\mu,0)=P$.
For $h=\partial_\mu x_\mu$ and $\mathcal L_\mu=\partial_t^2+1+\mu w-\rho+3x_\mu^{-4}$,
differentiating \eqref{pnlim}, we obtain $\mathcal L_\mu h=-wx_\mu$, while $\mathcal L_\mu x_\mu=4x_\mu^{-3}$.
Periodic integration by parts then produces
$
          \partial_\mu D_1(\mu,0)
        =\frac12\int_0^P wx_\mu^2\,dt>0.
$
The second integrand equals one at $\delta=0$, so its $x$-variation drops out of $\partial_\delta D_2$.
By differentiating directly and using \eqref{pnlim}, we find
$
         \partial_\delta D_2(\mu,0)
         =\frac12\int_0^P(\rho-\mu w)x_\mu^2\,dt.
$
Since $\partial_\mu D_2(\mu,0)=0$, we arrive at
\begin{equation}\label{pndet}
       \det\frac{\partial(D_1,D_2)}{\partial(\mu,\delta)}(\mu_*,0)
      =\frac14
         \left(\int_0^P wx_{\mu_*}^2\,dt\right)
          \left(\int_0^P(\rho-\mu_*w)x_{\mu_*}^2\,dt\right)<0.
\end{equation}
Shrink the rectangle so that the determinant remains nonzero and $\mu w>\rho$ on $\overline J\times[0,P]$.
\end{proof}

\begin{corollary}[Dense closed profiles and growing covers]\label{pdiv}
When $M_1$ and $M_2$ are closed, the members of Theorem \ref{pnonc} that close on finite $P$-covers are dense in the parameter family and give harmonic maps from closed source manifolds.
They can be selected with $(\mu_j,\delta_j)\to(\mu_*,0)$ while the least number of base periods required for closure tends to infinity.
\end{corollary}
\begin{proof}
The rank-two property and Proposition \ref{ccrit} provide dense rational returns.
Since $\partial_\delta D_2(\mu,0)<0$, choose such returns tending to $(\mu_*,0)$ with $D_2<P$.
Their reduced denominators tend to infinity: otherwise a subsequence would be constant and equal to its limit $P/(2\pi)$, contradicting $D_2<P$.
The least covering degrees therefore tend to infinity.
\end{proof}

\subsection{Persistence under changes of the first warp}

\begin{theorem}[Fixed-return deformation]
Fix $r\ge2$, $0<\alpha<1$, a positive smooth $P$-periodic $L_0$, and nondegenerate closed data $(\mu_*,\delta_*,x_*)$ from Corollary \ref{pbase}, with return $2\pi\rho$, $\rho\in\mathbb Q^2$.
Every positive $L$ sufficiently close to $L_0$ in $C^{r,\alpha}_{\rm per}$ has a unique nearby datum $(\mu(L),\delta(L),x(L))$ solving \eqref{pnex}, with the same return and $x(L)\in C^{r+2,\alpha}_{\rm per}$, depending smoothly on $L$.
The products close on the same $mP$-cover for every integer $m\ge1$ with $m\rho\in\mathbb Z^2$.
Their magnitudes and source warps remain bounded away from zero in a neighborhood of the chosen positive-$\delta_*$ datum.
\end{theorem}

\begin{proof}
Let $\mathcal R$ be the residual of \eqref{pnex}.
By Corollary \ref{pbase}, the operator $D_x\mathcal R$ is invertible.
The implicit-function theorem therefore first produces $x=X(L,\mu,\delta)$.
Since the reduced return has an invertible $(\mu,\delta)$-derivative, a second application fixes the value $2\pi\rho$.
Positivity and regularity persist, and Proposition \ref{ccrit} supplies closure.
\end{proof}

The fixed-return locus has codimension two and consists of closed graph geodesics of the varying periodic metrics $\mathcal K$ (Remark \ref{hperi}).
For circle identity inputs, we obtain harmonic maps from compact source manifolds into $\Sph^3$.
For real equatorial sphere inputs, $G^*g_{\rm round}=dt^2+a^2g_{\Sph^p}+b^2g_{\Sph^q}$, so the closed products are harmonic immersions into $\Sph^{2p+2q+3}$.

\section{Fixed-metric closure via Neumann dynamics}\label{seneu}

In Section \ref{seper} we fixed one source warp and reconstructed the second together with the profile.
We then used rational phase matching to produce closed profiles, generally on finite covers.
We now fix the entire source metric.
In the sector $F^2T=\kappa$, the change $d\tau=dt/F$ converts the profile equation into one autonomous Neumann system.
Closure therefore becomes an orbit-selection problem: we choose initial data whose trajectory returns after one source period or on a finite cover.
We use two mechanisms below, namely rational relative equilibria on carrier tori and period-matched rotations on the invariant real circle.

\subsection{The Neumann reduction}

\begin{theorem}[Time-changed Neumann reduction]\label{ntime}
Let $F\in C^\infty(I,(0,\infty))$ and $\kappa\in\R$, set $T=\kappa/F^2$, define $\tau(t)=\int_{t_0}^tF(\xi)^{-1}\,d\xi$, and let $(\ell_1,\ell_2)$ be the pair in Lemma \ref{gsep}.
The product $G_\gamma$ is harmonic for $g_{F,\kappa}=dt^2+\ell_1^2g_1+\ell_2^2g_2$ if and only if $\gamma=\Gamma\circ\tau$, where $\Gamma:\tau(I)\to\Sph^3$ solves the Neumann equation
\begin{equation}\label{nneu}
     \nabla_\tau^{\Sph^3}\dot\Gamma
          =\kappa\operatorname{grad}_{\Sph^3}U(\Gamma),
        \qquad U(z_1,z_2)=|z_2|^2/2,
\end{equation}
with mechanical potential $-\kappa U$.
\end{theorem}
\begin{proof}
Since $\lambda_2-\lambda_1=T$, we have $\operatorname{grad}_{\Sph^3}V=T\operatorname{grad}_{\Sph^3}U$ and $\nabla_t^{\Sph^3}\gamma'+(F'/F)\gamma'=F^{-2}\nabla_\tau^{\Sph^3}\dot\Gamma$.
Substitution into \eqref{peule} proves its equivalence with \eqref{nneu}.
\end{proof}

Using the conserved Neumann energy $H_N=\tfrac12|\dot\Gamma|^2-\kappa U(\Gamma)$, we obtain
\[
       |\gamma'|=1
        \quad\Longleftrightarrow\quad
       F^2=2H_N+\kappa|\Gamma_2|^2.
\]
Thus unit speed in the original parameter selects the corresponding energy relation among the Neumann trajectories.

\subsection{Constant magnitudes and rational phase closure}

For a $P$-periodic $F$, we continue to use the notation $I_F$ introduced in Theorem D.
The simplest periodic Neumann orbits are relative equilibria with constant profile magnitudes.
Their balance condition depends on the angular frequencies only through $\omega_1^2-\omega_2^2$, rather than on the carrier radius.
This independence allows one fixed metric to support a continuous family of closed profiles on different Hopf tori.
For constant $a,b>0$ with $a^2+b^2=1$, the relative equilibrium
\[
      \Gamma(\tau)
          =(ae^{i(\omega_1\tau+\theta_1)},be^{i(\omega_2\tau+\theta_2)}),
       \qquad \omega_1^2-\omega_2^2=\kappa,
\]
solves \eqref{nneu} by projection onto $T_\Gamma\Sph^3$.
Conversely, the phase laws force every harmonic profile with constant positive magnitudes to have this form after $d\tau=dt/F$, with $F^2T=\omega_1^2-\omega_2^2$.
Moreover, $F^2|\gamma'|^2=a^2\omega_1^2+b^2\omega_2^2$, so a nonstationary member has unit speed only if $F$ is constant.

\begin{corollary}[Periodic constant-magnitude families]\label{nper}
Let $F>0$ be $P$-periodic.
For integers $(r_1,r_2)\ne(0,0)$, set $\kappa=4\pi^2(r_1^2-r_2^2)/I_F^2$.
The frequencies $\omega_i=2\pi r_i/I_F$ then satisfy $\omega_1^2-\omega_2^2=\kappa$, which is the balance relation above.
For every $0<a<1$,
\[
     \Gamma_a(\tau)
      =\bigl(ae^{2\pi i r_1\tau/I_F},
     \sqrt{1-a^2}\,e^{2\pi i r_2\tau/I_F}\bigr)
\]
gives a harmonic product on the same metric $g_{F,\kappa}$ on $(\R/P\Z)\times M_1\times M_2$.
If $r_1r_2\ne0$ and $\gcd(|r_1|,|r_2|)=1$, these are embedded closed curves with coprime winding numbers $(r_1,r_2)$ and least $\tau$-period $I_F$.
Their length is
\begin{equation}\label{ntlen}
          \mathcal L(\Gamma_a)
      =2\pi\sqrt{a^2r_1^2+(1-a^2)r_2^2}.
\end{equation}
They are pairwise noncongruent when $|r_1|\ne|r_2|$.
\end{corollary}
\begin{proof}
The balance relation, combined with Lemma \ref{gsep}, proves harmonicity on the same periodic metric.
Two values $\tau,\widetilde\tau$ have the same profile point exactly when $r_i(\tau-\widetilde\tau)/I_F\in\mathbb Z$ for $i=1,2$.
We obtain the least period and embeddedness from coprimality.
The strictly varying length for $r_1^2\ne r_2^2$ excludes congruence under round isometries and reparametrization.
\end{proof}

\begin{proposition}[Carrier-independent energy and graph action]
In the setting of Corollary \ref{nper}, set $\omega_i=2\pi r_i/I_F$, let $G_a$ be the corresponding harmonic product, and let $\mathfrak c_a(t)=(t,\Gamma_a(\tau(t)))$ be its graph.
Independently of $a$,
\[
     |dG_a|^2
       =\lambda_1+\frac{\omega_1^2}{F^2}
       =\lambda_2+\frac{\omega_2^2}{F^2},
       \qquad
       \mathcal A_{\mathcal K}(\mathfrak c_a)
       =\frac12\int_0^P
          \left(F\lambda_1+\frac{\omega_1^2}{F}\right)dt
       =\frac12\int_0^P
          \left(F\lambda_2+\frac{\omega_2^2}{F}\right)dt.
\]
Thus either every $G_a$ is an eigenmap or none is, and the former occurs exactly when $F$ is constant.
The graph geodesics have the same $\mathcal K$-length, whereas their round profile lengths \eqref{ntlen} vary with $a$ when $|r_1|\ne|r_2|$.
\end{proposition}

\begin{proof}
Combining the identity $F^2|\gamma'|^2=a^2\omega_1^2+(1-a^2)\omega_2^2$ with the balance relation, we obtain the energy formula.
Substitution into \eqref{hkmet} yields the action formula.
If $F$ is constant, then $T$ is constant and Lemma \ref{gsep} makes both warps, hence both $\lambda_i$, constant, so the family consists of eigenmaps.
Conversely, suppose $|dG_a|^2=\mu>0$ is constant and set $x=F^2$.
Then $\lambda_i=\mu-\omega_i^2/x$, and the volume identity becomes
\[
       (\mu x-\omega_1^2)^{k_1}(\mu x-\omega_2^2)^{k_2}
       =d_1^{k_1}d_2^{k_2}x^{k_1+k_2-1}.
\]
This is a nonzero polynomial equation in $x$: its left side has degree $k_1+k_2$, whereas the right side has smaller degree.
Its positive roots are discrete, so continuity forces $x$, and therefore $F$, to be constant.
\end{proof}

Exactly one zero winding gives a singly spiral family.
More generally, a relative equilibrium closes on a finite cover precisely when $I_F\omega_i/(2\pi)\in\mathbb Q$ for both $i$.
For fixed $\kappa$, such a profile exists exactly when $\kappa I_F^2/(4\pi^2)\in\mathbb Q$, since every rational $q$ equals $((q+1)/2)^2-((q-1)/2)^2$.

\subsection{Axis-crossing profiles and unbounded degree}

The preceding profiles remain on a fixed Hopf torus and avoid the coordinate axes.
A second fixed-metric mechanism uses the invariant real great circle.
Its Neumann orbits cross both axes, and by choosing their energy we fit an arbitrary number of rotations into one source period.
For the identity inputs on $\Sph^p$ and $\Sph^q$, a real profile takes values in $\R^{p+1}\oplus\R^{q+1}$ and hence in the smaller sphere $\Sph^{p+q+1}$.
The equatorial inclusion into the larger complex sphere remains available, but it is not used in the degree calculation.

\begin{theorem}[Unbounded degree on one fixed metric]\label{ndeg}
Let $p,q\ge2$ be even, with real identity inputs $f_1=\operatorname{id}_{\Sph^p}$, $f_2=\operatorname{id}_{\Sph^q}$.
Fix a positive smooth $P$-periodic $F$ and $\kappa\in\R$.
Set $S_P^1=\R/P\Z$.
For every integer $m\ge1$, the single metric $g_{F,\kappa}$ carries a harmonic map
\begin{equation}\label{ndmap}
     \begin{split}
     G_m:S_P^1\times\Sph^p\times\Sph^q&\longrightarrow\Sph^{p+q+1},\\
      (t,x,y)&\longmapsto
       (\cos\Theta_m(\tau(t))x,\sin\Theta_m(\tau(t))y)
\end{split}
\end{equation}
of degree $4m$, where $\Theta_m(\tau+I_F)=\Theta_m(\tau)+2\pi m$.
The target is the smallest real sphere containing the image.
\end{theorem}
\begin{proof}
On the real profile circle, \eqref{nneu} becomes $\ddot\Theta=\kappa\sin\Theta\cos\Theta$, with $\dot\Theta^2=2h+\kappa\sin^2\Theta$.
For $h>h_*:=\max\{0,-\kappa/2\}$, its rotation period is
\[
     \mathcal P_\kappa(h)
        =\int_0^{2\pi}\frac{du}{\sqrt{2h+\kappa\sin^2u}}.
\]
Differentiating under the integral, we find $\mathcal P_\kappa'(h)<0$.
The period decreases from $+\infty$ to $0$.
For $\kappa\ne0$ the divergence at $h_*$ is logarithmic, while $\mathcal P_0(h)=2\pi/\sqrt{2h}$.
We can therefore choose a unique energy $h_m$ for which $\mathcal P_\kappa(h_m)=I_F/m$.
The corresponding increasing orbit defines \eqref{ndmap}, and Theorem \ref{ntime} shows that the map is harmonic.

Choose positively oriented frames $(x,e_1,\ldots,e_p)$ and $(y,f_1,\ldots,f_q)$ in $\mathbb R^{p+1}$ and $\mathbb R^{q+1}$.
After ordering the target coordinates by the $x$-block and then the $y$-block, we compute
$$
     G_m^*\omega_{p+q+1}
      =(-1)^p\dot\Theta_m\cos^p\Theta_m\sin^q\Theta_m
        \,d\tau\wedge\omega_p\wedge\omega_q.
$$
Since $d\tau=dt/F>0$, this agrees with the product orientation defined by increasing $t$ and the standard orientations of the two spheres.
Using the relation $\Theta_m(\tau+I_F)=\Theta_m(\tau)+2\pi m$, the spherical join volume formula, and the evenness of $p,q$, we obtain
$
         \deg G_m
          =\frac{m\int_0^{2\pi}\cos^pu\sin^qu\,du}
          {\int_0^{\pi/2}\cos^pu\sin^qu\,du}
     =4m.
$
\end{proof}

\begin{corollary}[Unbounded-degree eigenmaps on one metric]
Let $p,q\ge2$ be even, let $P>0$, and choose constants $\ell_1,\ell_2>0$ such that $p/\ell_1^2=q/\ell_2^2=:\lambda$.
On the fixed product manifold
\[
       \left(S_P^1\times\Sph^p\times\Sph^q,
       \,dt^2+\ell_1^2g_{\Sph^p}+\ell_2^2g_{\Sph^q}\right),
\]
the maps
\[
       \begin{gathered}
       \widehat G_m(t,x,y)
          =\left(\cos\frac{2\pi mt}{P}\,x,
          \sin\frac{2\pi mt}{P}\,y\right),
          \qquad m=1,2,\ldots,\\
       \Delta\widehat G_m
          =-\left(\lambda+\frac{4\pi^2m^2}{P^2}\right)\widehat G_m,
          \qquad \deg\widehat G_m=4m.
       \end{gathered}
\]
Thus one fixed Riemannian product metric supports spherical eigenmaps of arbitrarily large degree.
\end{corollary}

\begin{proof}
The two coordinate blocks are eigenfunctions with eigenvalue $\lambda+4\pi^2m^2/P^2$, and their squared norms sum to one.
This is the constant-$F$, $\kappa=0$ case of Theorem \ref{ndeg}, whose degree calculation gives $\deg\widehat G_m=4m$.
\end{proof}

\begin{corollary}[Closed geodesics for one conic metric]
In the setting of Theorem \ref{ndeg}, let $\mathcal K_{F,\kappa}$ be the conic metric of Theorem \ref{hkrop}, descended to $S_P^1\times\Sph^3$.
The real profiles in Theorem \ref{ndeg} give infinitely many mutually distinct positively oriented closed $\mathcal K_{F,\kappa}$-geodesics that are not iterates of shorter closed geodesics.
If $c_m$ denotes the graph associated with $\Theta_m$, their lengths satisfy
\[
       \mathcal L_{\mathcal K}(c_m)
       \ge \frac12\int_0^{I_F}\dot\Theta_m^2\,d\tau
       \ge \frac{2\pi^2m^2}{I_F},
\]
and hence tend to infinity.
\end{corollary}

\begin{proof}
The graph of every periodic harmonic profile is a closed $\mathcal K_{F,\kappa}$-geodesic by Theorem \ref{hkrop} and Remark \ref{hperi}.
Its projection to $S_P^1$ has degree one, so it cannot be a nontrivial iterate of another closed geodesic.
The real profiles have different winding numbers $m$ on the invariant real profile circle and are therefore mutually distinct.
Since the potential $V$ is positive, the graph action is at least its kinetic part.
Cauchy--Schwarz, together with $\Theta_m(I_F)-\Theta_m(0)=2\pi m$, yields the stated lower bound.
\end{proof}

The real profile curves cross the axes, but both source warps stay positive and the domain remains a product.
They connect with Smith's harmonic joins and with classical equivariant harmonic-map reductions \cite{Sm75,BR92}.
Equatorial inclusion preserves harmonicity in a larger complex target, where the maps are null-homotopic.

\section{Prescribed-volume inverse reconstruction}\label{sevol}

In the second unit-speed inverse problem, we prescribe $F$ and reconstruct $T$.
We treat monotone and periodic volume factors separately.

\subsection{Scalar reconstruction and the monotone threshold}

Set $a=\cos s$, with $0<s<\pi/2$, and $u=a^2$.
The unit-speed condition then becomes
\begin{equation}\label{vscal}
       (s')^2=1-\frac{\varepsilon V_c(s)}{F^2},
       \qquad V_c(s)=\sec^2s+c^2\csc^2s,
\end{equation}
or
\begin{equation}\label{vquad}
          (u')^2=4\mathcal A(t,u),
          \qquad
          \mathcal A=u(1-u)-\frac{\varepsilon[1+(c^2-1)u]}{F^2}.
\end{equation}
For $c\ne0$, $V_c$ has minimum $(1+|c|)^2$ at $s_c=\arctan\sqrt{|c|}$, so pointwise feasibility is exactly
\begin{equation}\label{vflor}
        F\ge\sqrt\varepsilon(1+|c|).
\end{equation}
Since time reversal exchanges increasing and decreasing prescribed volume factors, we assume $F'>0$ in the monotone problem.
Under strict inequality, we denote the roots of $\varepsilon V_c(s)=F(t)^2$ by $L(t)<s_c<U(t)$.
They form the lower and upper boundaries of the admissible $s$-interval, with $L$ decreasing and $U$ increasing.
Equivalently, if $0<u_-<u_+<1$ are the two walls in the $u$-variable, then $\mathcal A(t,u)=(u_+-u)(u-u_-)$.
For $c=0$, take $s_c=0$.
The admissible interval is instead $0<s\le U(t)=\arccos(\sqrt\varepsilon/F(t))$ for $F(t)>\sqrt\varepsilon$, with $s=0$ excluded because $b=0$.

For fixed $(c,\varepsilon,\sigma)$, or equivalently for momenta $C_1=\sigma\sqrt\varepsilon$ and $C_2=cC_1$, Theorem \ref{grec} turns each smooth solution of \eqref{vscal} into a harmonic reconstruction that is unique up to constant phases.
Here $R=\varepsilon Q/(F^2a^2)>0$ and
\begin{equation}\label{vtrec}
        T=\frac1{a^2Q}-\frac{a+a''}{aR}.
\end{equation}
Away from turns, we use \eqref{gcmp} to obtain
\begin{equation}\label{vquot}
          T=-\frac{2F'}{Fu'}.
\end{equation}
At turns one uses \eqref{vtrec}.
If $(M_i,g_i)$ are complete, every entire reconstruction has complete source: $dt^2$ controls escape in time, and the warps are bounded away from zero on compact time intervals.

The global lower limit $F_-:=\lim_{t\to-\infty}F(t)=\inf_{\mathbb R}F(t)$ controls pointwise feasibility.

\begin{theorem}[Monotone threshold]\label{vmono}
Let $F>0$ be smooth on $\R$, with $F'>0$ and $F_->0$.
For fixed $c\in\R$ and $\varepsilon>0$, \eqref{vscal} has a smooth entire solution if and only if
\begin{equation}\label{vstrc}
         \sqrt\varepsilon(1+|c|)\le F_-.
\end{equation}
For $c\ne0$, there are unique entire solutions in the respective half-strips
$$
         s_c<s^U<U,
        \quad (s^U)'>0,
       \qquad
       L<s^L<s_c,
          \quad (s^L)'<0.
$$
Set $U_\pm=\lim_{t\to\pm\infty}U(t)$ and $L_\pm=\lim_{t\to\pm\infty}L(t)$.
The two branches satisfy
\[
       \lim_{t\to-\infty}s^U(t)=U_-,
       \qquad \lim_{t\to+\infty}s^U(t)=U_+,
       \qquad
       \lim_{t\to-\infty}s^L(t)=L_-,
       \qquad \lim_{t\to+\infty}s^L(t)=L_+.
\]
For $c=0$, there is exactly one entire solution $s^U$ with $0<s^U<U$ and $(s^U)'>0$, and its limits are $U_-$ and $U_+$.
If $F_+:=\lim_{t\to\infty}F(t)=\infty$, then $U_+=\pi/2$.
For $c\ne0$, also $L_+=0$.
\end{theorem}

\begin{proof}
Pointwise feasibility proves necessity.
Under \eqref{vstrc}, $F(t)>F_-$ at finite times, so $U'>0$ and, for $c\ne0$, $L'<0$.
Set $f(t,s)=(1-\varepsilon V_c(s)/F(t)^2)^{1/2}$.
We solve $s_j'=f(t,s_j)$ from $s_j(-j)=s_c$, using the smooth extension across $s=0$ when $c=0$.
A first contact with $U$ would give $(U-s_j)'=U'>0$ at a zero reached from above, which is impossible.
Thus $s_j$ exists for $t>-j$ with $0<s_j'\le1$.

By compactness, the bounded $1$-Lipschitz family has a locally uniform subsequential limit, and continuity of $f$ shows that the limit satisfies the integral equation.
Neither $s-s_c$ nor $U-s$ can vanish: its derivative would be nonzero at a minimum.
The limit is therefore smooth and strictly increasing.
For $c\ne0$, we repeat the construction with $s_j'=-f(t,s_j)$ and obtain the lower branch.

Each branch is monotone and bounded, so it has limits at both ends.
If a limit lay strictly inside the limiting admissible interval, then $|s'|$ would stay bounded below there, a contradiction.
Hence the upper branch tends to $U_\pm$ and the lower branch to $L_\pm$.
On each half-strip the relevant vector field is strictly decreasing in $s$, so the positive difference of two ordered entire solutions is nonincreasing.
Their common backward limit forces equality.
For $c=0$, differentiating \eqref{vscal} at a hypothetical turn would produce the contradiction $0=2\varepsilon V_0F'/F^3>0$.
A decreasing solution, whose speed increases as $s$ and $1/F$ decrease, reaches $s=0$ in finite time.
Thus the increasing branch is the only entire solution for $c=0$.
\end{proof}

\subsection{Finite and infinite volume-factor limits}

If $F_+<\infty$, both magnitudes approach positive limits and determine a nondegenerate limiting carrier torus, but the accumulation set still depends on the phases.
At least one limiting phase velocity is nonzero, so the profile cannot converge to a point.
Only when $F_+=\infty$ are we forced to have one magnitude vanish, in which case the limiting torus collapses to a coordinate circle.

\begin{remark}[Finite-limit phase geometry]
Assume that $a(t)\to a_*\in(0,1)$ and $(s_1',s_2')\to\omega_*=(\omega_{1,*},\omega_{2,*})$, with both limiting velocities nonzero.
If $\omega_{2,*}/\omega_{1,*}$ is irrational, the accumulation set of the profile curve is the whole limiting carrier torus.
Due to convergence of the magnitude and phase velocities, every sufficiently late segment of any fixed duration is uniformly close to the corresponding segment of the limiting linear flow.
We choose a sufficiently long limiting segment that is $\delta$-dense and then let $\delta\downarrow0$, which proves the claim.
If instead $\omega_*=\rho(r_1,r_2)$ for a coprime integer pair, set $\psi(t)=r_2s_1(t)-r_1s_2(t)\pmod{2\pi}$.
The $\omega$-limit set of the profile is the union of the closed $(r_1,r_2)$-orbits indexed by the limit set of $\psi$.
It is one closed orbit if $\psi$ converges, for example when $\int_T^\infty |r_2s_1'(t)-r_1s_2'(t)|\,dt<\infty$.
Thus a rational limiting slope alone is insufficient: without transverse phase locking, several parallel closed orbits, or even the whole torus, may occur.
Both the phase-locked rational case and the irrational case occur in Theorem \ref{grec} by prescribing an admissible magnitude with $a-a_*,a',a''=O(e^{-t})$.
The reconstructed volume factor and phase velocities then approach their limits with integrable errors.
\end{remark}

\begin{corollary}[Large-volume axis geometry]\label{vaxis}
Assume Theorem \ref{vmono} with $F(t)\to\infty$ as $t\to+\infty$.
Set $a_U=\cos s^U$, $b_U=\sin s^U$ and, when $c\ne0$, set $a_L=\cos s^L$, $b_L=\sin s^L$.
On the upper branch $a_U\to0$ and $b_U\to1$.
Its surviving phase is
\begin{equation}\label{vphup}
        s_2(t)=s_2(t_0)+\sigma c\sqrt\varepsilon
         \int_{t_0}^t\frac{d\xi}{F(\xi)b_U(\xi)^2}.
\end{equation}
For $c\ne0$, the lower branch has $a_L\to1$ and $b_L\to0$.
Its surviving phase is
\begin{equation}\label{vphdn}
         s_1(t)=s_1(t_0)+\sigma\sqrt\varepsilon
      \int_{t_0}^t\frac{d\xi}{F(\xi)a_L(\xi)^2}.
\end{equation}
On either branch with nonzero surviving momentum, $\int_{t_0}^{\infty}F(t)^{-1}\,dt<\infty$ gives convergence to one point of the limiting coordinate circle, whereas divergence of this integral gives that whole circle as the accumulation set of the profile curve.
When $c=0$, the upper branch always converges to a point since $s_2'=0$.
For $c\ne0$ the two branches have the same prescribed $F$ and momentum data but different reconstructed source metrics, with $T_U>0$ and $T_L<0$.
\end{corollary}
\begin{proof}
Theorem \ref{vmono} provides the magnitude limits, while \eqref{gpha} provides the phase formulas.
Since $b_U^2,a_L^2\to1$, the phase integrals in \eqref{vphup} and \eqref{vphdn} converge exactly when $\int_{t_0}^{\infty}F^{-1}\,dt$ does.
Otherwise its phase is strictly monotone and unbounded, and visits every angle arbitrarily late while the other profile component tends to zero.
Finally, for $u_U=a_U^2$ and $u_L=a_L^2$, one has $u_U'<0<u_L'$.
Applying \eqref{vquot}, we obtain the asserted signs of $T$.
\end{proof}

The three end regimes are summarized in Figure \ref{fvol}.

We can realize the two large-volume outcomes with the same lower bound by taking $A>\sqrt\varepsilon(1+|c|)$ and $c\ne0$.
The functions $F_{\rm fast}(t)=A+e^t$ and $F_{\rm slow}(t)=A+\sqrt{\log(1+e^t)}$ are smooth, strictly increasing and tend to infinity.
Their reciprocals are respectively integrable and nonintegrable, so they realize the two alternatives of Corollary \ref{vaxis}.

The exact and asymptotic finite-height regimes also arise directly from the prescribed-magnitude reconstruction.
For $c=\varepsilon=1$, set
$
       F[a]=\frac1{a\sqrt{1-a^2-(a')^2}}.
$
The $C^\infty$-flat magnitude $a_{\rm flat}$ in Section \ref{seglb} is nonconstant before its joining time and constant afterward, so it reaches a fixed carrier torus smoothly and follows a closed or dense orbit according to whether its constant phase-slope ratio is rational or irrational.
If instead we take $a(t)=a_*+\eta e^{-rt}$, with $0<a_*<1$ and small $\eta,r>0$, we obtain $F[a]\to F_\infty<\infty$, exponential phase convergence, and the alternatives H and I.

In G--I, the colored arcs remain off the limiting torus and approach its gray limiting motion; phase locking in H gives one closed limiting orbit, whereas the irrational limiting slope in I gives the whole torus.
Black and blue match each volume graph with its spatial panels, and dashed curve segments lie behind the pale carrier torus.

\begin{figure}[H]
\centering
\includegraphics[width=.95\linewidth,height=.82\textheight,keepaspectratio]{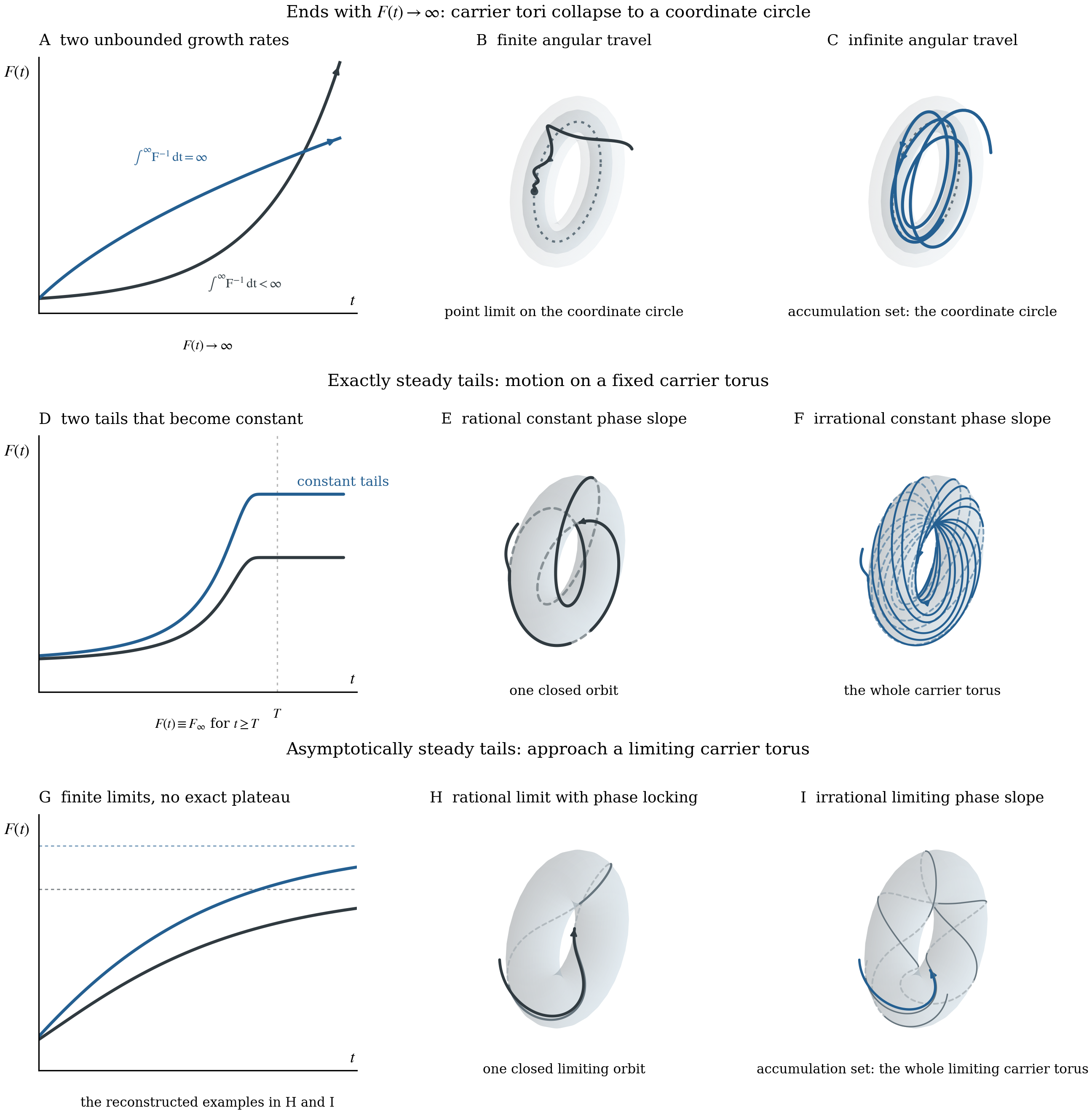}
\caption{Unbounded-volume ends (A--C), exact joins to steady tails (D--F), and asymptotically steady finite-height ends (G--I).}
\label{fvol}
\end{figure}

\subsection{Periodic volumes}

For a nonconstant periodic magnitude, we must separate scalar return from profile closure.
The magnitude first turns at the moving walls and returns after one period.
We then apply Proposition \ref{ccrit}: the profile closes on a finite cover precisely when both phase increments are rational multiples of $2\pi$.
A smooth turn can occur only when the magnitude meets a moving wall at a critical time of $F$.
We establish the required local analytic continuation in Appendix \ref{apvol}.

\begin{theorem}[A periodic magnitude for a prescribed volume]\label{vwallcap}
Let $F>0$ be analytic and $2H$-periodic, where $0<H<\pi/2$.
Assume it is even about $0$ and $H$, with $F'<0$ on $(0,H)$ and $F''(0)<0<F''(H)$.
Set $M=F(0)$, $m=F(H)$, and $h=(1-m^2/F^2)^{1/2}$.
If
\begin{equation}\label{vwcond}
          1-\frac{m^2}{M^2}+\frac{m^2F''(0)}{M^3}>0,
     \qquad h(0)>2\int_0^Hh(t)\,dt,
\end{equation}
then for some $\alpha\in(0,m)$ the equation
\begin{equation}\label{vweq}
         (s')^2=1-\frac{\alpha^2}{F^2\sin^2(2s)}
\end{equation}
has a nonconstant analytic $2H$-periodic solution in $(0,\pi/2)$.
Thus \eqref{vscal} with $c=1$ and $\varepsilon=\alpha^2/4$ has a periodic magnitude and positive analytic periodic warps.
\end{theorem}

\begin{proof}
We prove the result by shooting, terminal matching, and reflection in Appendix \ref{apvol}.
\end{proof}

For example, the hypotheses hold for $F(t)=1+\tfrac92(1+\cos(3\pi t))$ and $H=1/3$.
The theorem produces a periodic magnitude, but profile closure still requires rational phase return.
The following family realizes both conditions.

\begin{proposition}[An explicit periodic closing family]
For every $\nu>0$ there is an analytic family, parametrized by $0<\delta<\pi/4$, with periodic magnitude, positive periodic warps, and equal phase increments $\Theta(\delta)$.
The parameters for which the profile closes on a finite cover are dense.
There is a sequence $\delta_j\downarrow0$ whose least closing multipliers tend to infinity.
\end{proposition}

\begin{proof}
Set $c=1$, $\varepsilon=\nu^2$, $\sigma=1$, and
$
         \theta_\delta=\frac\pi4+\delta\sin t,
          \, a_\delta=\sin\theta_\delta,
        \,
       F_\delta=
      \frac{2\nu}{\cos(2\delta\sin t)\sqrt{1-\delta^2\cos^2t}}.
$
Here $R=\cos^2\theta_\delta(1-\delta^2\cos^2t)>0$, and hence \eqref{greg} produces an entire reconstruction with positive periodic warps.
The volume is $\pi$-periodic, while the magnitude has period $2\pi$.
Translation by $\pi$ interchanges the two phase derivatives, whose increments over $2\pi$ have the common value
$
     \Theta(\delta)=\int_0^{2\pi}
      \frac{\sqrt{1-\delta^2\cos^2t}}{\cos(2\delta\sin t)}\,dt.
$
The function $\Theta$ is even and analytic for $|\delta|<\pi/4$, with $\Theta(\delta)=2\pi(1+\tfrac34\delta^2+O(\delta^4))$.
It follows that the parameters for which $\Theta(\delta)/(2\pi)$ is rational are dense, and each such parameter closes the profile on a finite cover.
For large $j$, we choose $\Theta(\delta_j)/(2\pi)=1+1/j$.
Then $\delta_j\downarrow0$, while the least number of $2\pi$-periods required for closure tends to infinity.
\end{proof}

\begin{figure}[H]
\centering
\includegraphics[width=.85\linewidth]{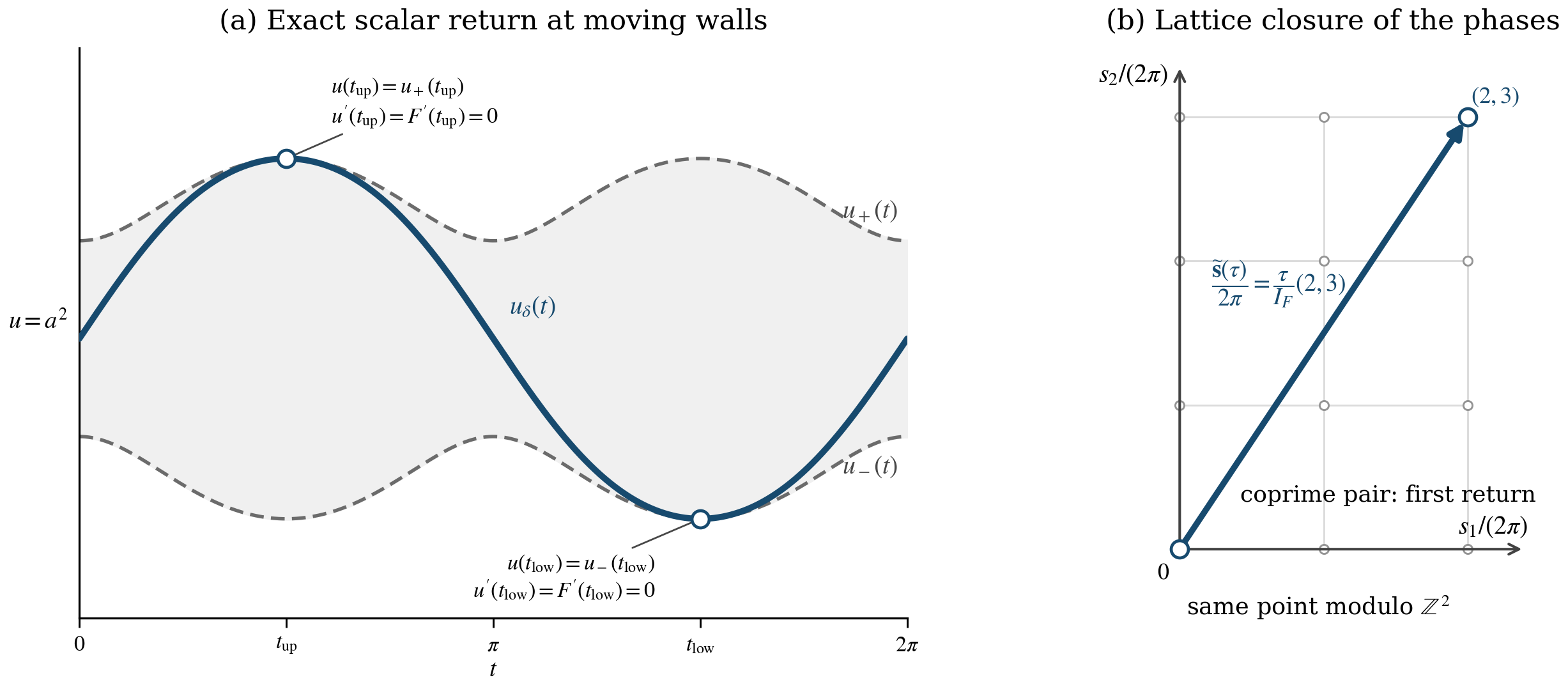}
\caption{The two conditions for periodic closure: return of the magnitude between its moving walls and rational return of the lifted phases.}
\label{fgeom}
\end{figure}

Figure \ref{fgeom} separates the two closure conditions.
The magnitude first returns between the moving walls $u_-$ and $u_+$.
The lifted phases must then have rational displacement; an integral displacement closes on the base period, while a nonintegral rational displacement closes on a finite cover.

The obstruction in Proposition \ref{vobcritprop} shows that pointwise feasibility does not guarantee a global periodic reconstruction.
For example, it excludes $(c,\varepsilon)=(1,1)$ for the positive periodic volume $F(t)=2.0001+1.9999\sin^2(4\pi t)$.
\section{Concluding remarks}

We have separated inverse reconstruction, where the source metric may vary, from fixed-source-metric closure, where an autonomous Neumann reduction selects closed profiles on one source.
The conic graph metric provides their common variational layer.
In the fixed-metric sector, we use relative equilibria to obtain continuous closed families and real rotations to obtain maps of unbounded degree.
For suitably matched constant warps, these maps are eigenmaps.
It remains to characterize closed profiles for general fixed warp pairs beyond the sector $F^2T=\mathrm{constant}$.
Outside this autonomous sector, the profile equation is nonautonomous and should admit further closure mechanisms.

\appendix

\section{Finite-regularity extensions}

For locally bounded positive warps with locally bounded reciprocals, we call $G$ \emph{weakly harmonic} if
\begin{equation}\label{gweak}
     \int\langle dG,dV\rangle_g\,d\mathrm{vol}_g=0
\end{equation}
for every compactly supported tangent variation $V$ along $G$.
The volume density and inverse metric are locally bounded, so this definition applies even when the source metric is only measurable.

\begin{proposition}[Nonsmooth inverse data]
Fix $(c,\varepsilon,\sigma)$ as in Theorem \ref{grec}.
\begin{enumerate}
\item If $a\in C^{1,1}_{\mathrm{loc}}(I,(0,1))$ and $R=1-a^2-(a')^2>0$, equations \eqref{greg}--\eqref{gpha}, with $a''$ taken almost everywhere, define a locally uniformly elliptic measurable source metric and a unit-speed weakly harmonic product.
Here $F\in W^{1,\infty}_{\mathrm{loc}}$ and $T\in L^\infty_{\mathrm{loc}}$.
If $a\in C^{2,\alpha}_{\mathrm{loc}}$ for $0<\alpha<1$, the source metric is $C^{0,\alpha}_{\mathrm{loc}}$.
\item If the prescribed first warp $\ell_1>0$ is continuous on $\R$ and $c\ne0$, the global existence and uniqueness of Theorem \ref{gwar} still hold, with $a\in C^2$, $F\in C^1$, and a positive continuous second warp.
The product is weakly harmonic.
If $\ell_1$ is locally Lipschitz, so is the source metric.
\item If $F\in C^1(\R,(0,\infty))$, $F'>0$, and $F_->0$, the threshold and the selected entire branches of Theorem \ref{vmono} remain valid.
Their reconstructed metrics are continuous, and the products are weakly harmonic.
If $F\to\infty$, the limit alternatives of Corollary \ref{vaxis} hold as well.
\end{enumerate}
\end{proposition}
\begin{proof}
In (1), $a'$ is locally Lipschitz, so $F$ is locally Lipschitz and $T$ is locally bounded.
We apply Lemma \ref{gsep} to obtain positive warps with locally bounded reciprocals.
The identities \eqref{gder}--\eqref{grad} hold almost everywhere and, combined with the conserved phases, yield \eqref{peule} almost everywhere.
Integrating by parts in $t$ and along the smooth eigenmap inputs proves \eqref{gweak} for arbitrary target variations, with the reconstructed metric held fixed.
The same formulas establish the stated regularity.
For (2), the vector field \eqref{gvec} is continuous in time and locally Lipschitz in its state.
The compact-interval bounds from \eqref{gvolw} require only upper and lower bounds on $\ell_1$, so we can repeat the continuation proof.
In (3), the wall barriers and comparison from Theorem \ref{vmono} involve only $F$ and $F'$.
Each selected solution stays strictly between its walls and is therefore $C^2$, while \eqref{vtrec} determines a continuous $T$.
In both (2) and (3), the profile equation verifies \eqref{gweak}.
\end{proof}

\section{Periodic reconstruction and nondegenerate closed data}\label{apcl}

The arguments in Section \ref{seper} require two periodic existence results.
We first construct small-momentum magnitude branches for an arbitrary prescribed periodic warp and control their parameter dependence.
We then produce a nondegenerate periodic Pinney datum for the uniformly positive-warp branch and the fixed-return deformation.

\subsection{Small-momentum reconstruction}

Let $L>0$ be smooth and $P$-periodic, set $k=k_2$, and define
\[
       y_0=d_2^{-k/(2k+2)}L^{-k_1/(k+1)},\qquad q=(2+2/k)y_0^{-2},\qquad
          S=\int_0^P\sqrt q\,dt,\qquad \delta_n=\frac{S}{2\pi(n+\tfrac12)}.
\]

\begin{lemma}[Periodic spectral estimate]\label{pspc}
Set $m=y_0^{-2/k}$ and $\mathcal L_\delta h=\delta^2m^{-1}(mh')'+qh$.
For all large $n$, $\mathcal L_{\delta_n}:H^2_{\rm per}\to L^2_{\rm per}$ is invertible, with
$$
        \|\mathcal L_{\delta_n}^{-1}f\|_{X_{\delta_n}}
         \le C\delta_n^{-1}\|f\|_2,
       \qquad
         \|h\|_{X_\delta}=\|h\|_2+\delta\|h'\|_2+\delta^2\|h''\|_2,
$$
where $C$ is independent of $n$.
\end{lemma}
\begin{proof}
On periodic $L^2(mq\,dt)$, let $\mathscr A h=-(mq)^{-1}(mh')'$.
The unitary Liouville transformation $s(t)=\int_0^t\sqrt q\,dt$, $z=m^{1/2}q^{1/4}h$ sends $\mathscr A$ to $-\partial_s^2+V$ on the circle of length $S$, with bounded $V$.
By comparing its spectrum with the free spectrum $\{(2\pi j/S)^2:j\in\mathbb Z\}$ through the min--max principle, we obtain $\operatorname{dist}(1,\delta_n^2\operatorname{spec}\mathscr A)\ge c\delta_n$.
Since $\mathcal L_{\delta_n}=q(I-\delta_n^2\mathscr A)$, the spectral theorem provides the $L^2$ bound, while the equation and interpolation control the two derivative terms.
\end{proof}

\begin{theorem}[Periodic reconstruction]\label{pall}
For every $c_*\in\mathbb R$ and all large $n$, there is a smooth local family $(c,\delta)\mapsto y\in H^2_{\rm per}$ near $(c_*,\delta_n)$ giving admissible magnitudes $a=\delta y$ whose reconstruction has $\ell_1=L$ and $\varepsilon=\delta^{2(k+1)}$.
At the base points,
\begin{equation}\label{pasy}
     \frac{a_n}{\delta_n}\longrightarrow y_0\quad\hbox{in }C^1,
      \qquad
         \frac{\ell_{2,n}}{\delta_n}\longrightarrow\sqrt{d_2}\,y_0,
          \qquad
       \frac{F_n}{\delta_n^k}\longrightarrow y_0^{-1},
\end{equation}
the last two limits being uniform, and
\begin{equation}\label{pfbd}
          \|\partial_\delta y\|_{X_{\delta_n}}=O(1),
      \qquad
     \|\partial_c y\|_{X_{\delta_n}}=O(\delta_n).
\end{equation}
Every resulting magnitude is smooth.
The parameter neighborhoods may shrink with $n$.
\end{theorem}

\begin{proof}
Set $A=d_2L^{2k_1/k}$, $D=d_1/L^2$, $R_\delta=1-\delta^2(y^2+y'^2)$ and $Q_\delta=1+(c^2-1)\delta^2y^2$.
Under the stated scaling, \eqref{gvec} becomes $\mathcal H_\delta(y)=0$, where
\[
       \mathcal H_\delta(y)
          =\delta^2y''-\frac{R_\delta}{yQ_\delta}
         +Ay^{1+2/k}R_\delta^{1+1/k}Q_\delta^{-1/k}
        -\delta^2y(DR_\delta-1).
\]
At $\delta=0$ we find the unique positive root $y_0$, whose linearization is $q$.
At $(y_0,y_0')$, the nondifferential part $N(t,y,p,\delta,c)$, with $p=y'$, satisfies $N_y=q+O(\delta^2)$ and $N_p=-(2/k)\delta^2y_0'/y_0+O(\delta^4)$, which accounts for the drift in $\mathcal L_\delta$.
There is a smooth periodic $E_0$ such that $\mathcal H_\delta(y_0)=\delta^2E_0+O(\delta^4)$.
Set $y_\delta^*=y_0-\delta^2E_0/q$.
Then
$$
         \mathcal H_\delta(y_\delta^*)=O_{L^2}(\delta^4),
       \qquad
          D\mathcal H_\delta(y_\delta^*)
          =\mathcal L_\delta+O(\delta^2)+O(\delta^4)\partial_t.
$$
The error has $X_\delta\to L^2$ norm $O(\delta^2)$.
Lemma \ref{pspc}, followed by a Neumann-series argument, yields $\|B_n^{-1}\|_{L^2\to X_{\delta_n}}=O(\delta_n^{-1})$ for $B_n=D\mathcal H_{\delta_n}(y_{\delta_n}^*)$.

On an $X_\delta$-ball of radius $r$, scaled Sobolev embedding and smoothness in $(t,y,\delta y',\delta^2,c)$ provide
\[
        \|h\|_\infty+\|\delta h'\|_\infty\le C\delta^{-1/2}\|h\|_{X_\delta},\qquad
      \|N_\delta(h)-N_\delta(\widetilde h)\|_2\le C\delta^{-1/2}r\|h-\widetilde h\|_{X_\delta}.
\]
where $N_\delta$ is the Taylor remainder at $y_\delta^*$.
For $r=M\delta_n^3$ with fixed large $M$, the map $h\mapsto-B_n^{-1}(\mathcal H_{\delta_n}(y_{\delta_n}^*)+N_{\delta_n}(h))$ preserves the ball and contracts with Lipschitz constant $O(\delta_n^{3/2})$.
Its fixed point satisfies $\|h\|_\infty=O(\delta_n^{5/2})$, $\|h'\|_\infty=O(\delta_n^{3/2})$ and $\|h''\|_2=O(\delta_n)$.
Thus $y_n=y_{\delta_n}^*+h\to y_0$ in $C^1$ and $a_n=\delta_ny_n$ is admissible.
Bootstrapping proves smoothness.
Substituting into the identities $F_n=\delta_n^k\sqrt{Q_{\delta_n}}/(y_n\sqrt{R_{\delta_n}})$ and $\ell_{2,n}=(F_n/L^{k_1})^{1/k}$ then proves \eqref{pasy}.

The exact linearization differs from $B_n$ by $O(\delta_n^{5/2})$ in $X_{\delta_n}\to L^2$ and retains the $O(\delta_n^{-1})$ inverse bound.
The periodic implicit-function theorem now produces the required families.
At the base points $\|y''\|_2+\|y'\|_\infty=O(1)$.
In the variable $z=\delta y'$, $N$ is even in $z$ and depends on $c$ only through $Q_\delta$.
Hence $N_z=O(\delta)$ and $N_c=O(\delta^2)$.
After including the term $2\delta y''$, we obtain $\|\partial_\delta\mathcal H_\delta(y)\|_2=O(\delta)$ and $\|\partial_c\mathcal H_\delta(y)\|_2=O(\delta^2)$.
Implicit differentiation then proves \eqref{pfbd}.
\end{proof}

\subsection{The periodic Pinney limit}

\begin{lemma}\label{pnell}
For smooth $P$-periodic $u,w$ with $w>0$, there are arbitrarily large $\mu>0$ for which the equation $y''+(1+\mu w-u)y=0$ has strictly elliptic monodromy.
At every such $\mu$, the Pinney equation $x''+(1+\mu w-u)x=x^{-3}$ has a positive smooth periodic solution, and $\mathcal L=\partial_t^2+1+\mu w-u+3x^{-4}$ is an isomorphism $C^{2,\alpha}_{\rm per}\to C^{0,\alpha}_{\rm per}$ for $0<\alpha<1$.
\end{lemma}
\begin{proof}
Under the Liouville transformation $s(t)=\int_0^t\sqrt w\,dt$, $S=s(P)$ and $y=w^{-1/4}z$, the equation becomes $z_{ss}+(\mu+V)z=0$, where $V$ is bounded and periodic and the monodromy is conjugate to the original one.
Variation of constants yields $\operatorname{tr}M_\mu=2\cos(\sqrt\mu S)+O(\mu^{-1/2})$.
We therefore choose $\sqrt{\mu_n}S=(n+\tfrac12)\pi$ and obtain strict ellipticity for large $n$.

Fix such a parameter, set $p=1+\mu w-u$, and take the fundamental matrix with $Y(0)=I$ and $M=Y(P)$.
Since $M$ is conjugate to a rotation, choose $C>0$ with $MCM^T=C$ and $\det C=1$.
For the first row $r$ of $Y$, the Pinney formula $x=\sqrt{rCr^T}$ \cite{Pi50} is positive and periodic.
Using $r''=-pr$ and the unit Wronskian, we verify that $x''+px=x^{-3}$.
Set $\theta(t)=\int_0^t x^{-2}\,dt$ and $\Theta=\theta(P)$.
The fundamental pair $x\cos\theta,x\sin\theta$ yields $\operatorname{tr}M=2\cos\Theta$, so $\Theta\notin\pi\mathbb Z$.
Since $\mathcal L(xf(\theta))=x^{-3}(f_{\theta\theta}+4f)$, the linearized monodromy is conjugate to rotation through $2\Theta$ and has no fixed vector.
Since the periodic operator has Fredholm index zero, we conclude that it is invertible.
\end{proof}

\begin{corollary}[Nondegenerate closed data]\label{pbase}
Every positive smooth $P$-periodic $L_0$ admits closed data $(\mu_*,\delta_*,x_*)$ from Corollary \ref{pdiv}, with arbitrarily small $\delta_*>0$, for which the magnitude-equation linearization $D_x\mathcal R$ of \eqref{pnex} and the reduced phase-increment derivative in $(\mu,\delta)$ are invertible.
\end{corollary}
\begin{proof}
The invertibility of the magnitude linearization from Lemma \ref{pnell} persists for small positive $\delta$, and \eqref{pndet} gives invertibility of the phase derivative.
Rational returns are dense in the open phase images.
Applying Proposition \ref{ccrit}, we obtain closed data, while elliptic regularity supplies inverses on all periodic H\"older scales.
\end{proof}

\section{Wall matching and a periodic obstruction}\label{apvol}

In this appendix we prove two facts used in Section \ref{sevol}: analytic continuation through a wall contact and a computable obstruction to global periodic reconstruction.

\subsection{Analytic wall germs and terminal matching}

\emph{Turning condition.}
Let $c\ne0$, assume strict inequality in \eqref{vflor}, and suppose all critical points of $F$ are nondegenerate.
If a smooth solution of \eqref{vquad} turns at $t_0$, then $F'(t_0)=0$ and $u_0=u(t_0)$ is one of the two walls.
For $B_0=1-2u_0-\varepsilon(c^2-1)/F(t_0)^2$ and $q=u''(t_0)$, one has
\begin{equation}\label{vmjeq}
      q^2-2B_0q-4u_0(1-u_0)\frac{F''(t_0)}{F(t_0)}=0,
         \qquad T(t_0)=-\frac{2F''(t_0)}{F(t_0)q}.
\end{equation}
Thus $q\ne0$.
If $F$ is analytic near $t_0$, each real candidate satisfying $2B_0/q-2\notin\{1,2,\ldots\}$ determines a unique analytic germ with the prescribed $2$-jet.

\begin{proof}
Differentiating \eqref{vquad} at a turn, we obtain $\mathcal A=\mathcal A_t=0$ and hence $F'=0$.
A second differentiation yields \eqref{vmjeq}, and \eqref{vquot} determines $T(t_0)$.
Let $x=t-t_0$ and $u=u_0+x^2v$.
Since $\mathcal A=\mathcal A_t=0$ at $(t_0,u_0)$, $\widetilde{\mathcal A}(x,v)=x^{-2}\mathcal A(t_0+x,u_0+x^2v)$ is analytic, with $\widetilde{\mathcal A}(0,q/2)=q^2/4>0$ by \eqref{vmjeq}.
The equation with the prescribed $2$-jet is equivalent to
$$
          xv'=2\operatorname{sgn}(q)\sqrt{\widetilde{\mathcal A}(x,v)}-2v=:\Phi(x,v),
          \qquad v(0)=q/2.
$$
Here $\Phi(0,q/2)=0$ and $\Phi_v(0,q/2)=\rho(q)$.
Under the stated condition $\rho(q)\notin\{1,2,\ldots\}$, we apply the Briot--Bouquet theorem to obtain a unique analytic germ \cite[Chapter 12, \S12.1]{Hil76}.
The coefficient recursion, or equivalently the analytic implicit-function argument, also provides analytic dependence on the auxiliary analytic parameters.
Moreover, $\mathcal A=x^2\widetilde{\mathcal A}>0$ off $x=0$.
\end{proof}

\begin{lemma}[Terminal wall matching]\label{vwterm}
Let $A$ be analytic near $H$, with $A(H)>0$ and $A'(H)=0<A''(H)$.
Any strictly decreasing solution of $(z')^2=4(A-z^2)$ arriving from $t<H$ at $z_*=z(H)$, $z_*^2=A(H)$, extends analytically across $H$, with
\[
        z''(H)=-2z_*+\sqrt{4z_*^2+2A''(H)}>0.
\]
If $A$ is even about $H$, so is the extension.
\end{lemma}

\begin{proof}
Let $x=H-t$, $w(x)=z(H-x)-z_*$, and $A(H-x)-A(H)=x^2B(x)$, where $\beta=B(0)=A''(H)/2>0$.
Then $w\ge0$ is increasing and
$$
     w'=2\sqrt{x^2B(x)-2z_*w-w^2}\le C(x+\sqrt w).
$$
After integration, we have $w(x)\le Cx^2+Cx\sqrt{w(x)}$, and hence $w=O(x^2)$.
For $\eta=-z_*+\sqrt{z_*^2+\beta}>0$, the Briot--Bouquet equation
$$
          xv'=2\sqrt{B(x)-2z_*v-x^2v^2}-2v,
      \qquad v(0)=\eta,
$$
has linear coefficient $-2(z_*+\eta)/\eta<0$, hence a unique analytic germ $\bar v$ and an analytic candidate $\bar w=x^2\bar v=\eta x^2+O(x^3)$.

Set $R(x,w)=x^2B(x)-2z_*w-w^2$.
The difference $e=w-\bar w$ satisfies
$$
          e'=\kappa e,
     \qquad
     \kappa=-\frac{2(2z_*+w+\bar w)}
     {\sqrt{R(x,w)}+\sqrt{R(x,\bar w)}}.
$$
Since $\sqrt{R(x,\bar w)}=\eta x+O(x^2)$ and $w,\bar w\ge0$,
$$
         \limsup_{x\downarrow0}x\kappa
          \le\frac{4\max\{-z_*,0\}}\eta<2.
$$
Choose $K<2$ with $\kappa\le K/x$ near zero.
A nonzero $e$ satisfies $|e(x)|\ge Cx^K$, contrary to $e=O(x^2)$.
Thus $w=\bar w$, from which the formula $z''(H)=2\eta$ follows.
When $A$ is even, reflection preserves the $2$-jet; uniqueness of the analytic germ then proves evenness.
\end{proof}

\begin{proof}[Proof of Theorem \ref{vwallcap}]
With $z=\sin(2s-\pi/2)$ and $A_\alpha=1-\alpha^2/F^2$, \eqref{vweq} becomes
\[
        (z')^2=4(A_\alpha-z^2).
\]
For $0<\alpha<m$ the walls $\pm\sqrt{A_\alpha}$ move inward on $(0,H)$.
Start at the upper wall $z_{0,\alpha}=\sqrt{A_\alpha(0)}$ with the fast inward jet
\[
         q_\alpha=-2z_{0,\alpha}-2\sqrt{D_\alpha},
          \qquad
       D_\alpha=1-\frac{\alpha^2}{M^2}+\frac{\alpha^2F''(0)}{M^3}.
\]
By \eqref{vwcond}, $D_\alpha>0$ for $0\le\alpha\le m$.
After the substitution $z=z_{0,\alpha}+t^2v$, the Briot--Bouquet linear coefficient is $-4z_{0,\alpha}/q_\alpha-2\in[-1,0)$.
The argument above therefore provides a unique analytic germ depending analytically on $\alpha$.
Continue it by $z'=-2\sqrt{A_\alpha-z^2}$ until first contact.
At $\alpha=0$ the solution is $\cos(2t)$.
Since $H<\pi/2$, continuous dependence ensures strict survival on $(0,H]$ for small $\alpha>0$.

Strict survival through $H$ is open: a branch with positive gap at $H$ has positive gap on the intervening compact interval, by analytic dependence at $0$ and ordinary continuous dependence elsewhere.
Let $(0,\alpha_*)$ be its component issuing from zero.
As $\alpha_j\uparrow\alpha_*$, the bounds $|z_{\alpha_j}|\le1$ and $|z_{\alpha_j}'|\le2$ provide a uniformly convergent subsequence on $[0,H]$, whose limit satisfies
$$
         z(t)=z(0)-2\int_0^t\sqrt{A_{\alpha_*}(\xi)-z(\xi)^2}\,d\xi.
$$
Near zero this limit agrees with the selected germ by analytic dependence.
The gap $\mathscr D=A_{\alpha_*}-z^2\ge0$ cannot vanish inside $(0,H)$: at such a minimum, $\mathscr D'=A_{\alpha_*}'<0$, which is impossible.
If $\alpha_*=m$, the endpoint gives $z(H)=0$, and hence
$$
         h(0)=\int_0^H-z'(t)\,dt
         \le2\int_0^Hh(t)\,dt,
$$
contrary to \eqref{vwcond}.
Thus $\alpha_*<m$.
If the limiting branch had positive gap at $H$, survival openness would extend the component beyond $\alpha_*$.
Thus the first contact occurs exactly at $H$.

At $H$, we have $A_{\alpha_*}>0$ and $A_{\alpha_*}'=0<A_{\alpha_*}''$, so Lemma \ref{vwterm} provides analytic terminal matching.
The evenness of $F$, together with uniqueness of the selected initial jet, proves evenness at $0$; the terminal matching lemma proves evenness at $H$.
We may therefore reflect the solution to obtain a $2H$-periodic analytic solution.
Since $\alpha_*<m$, $|z|<1$ and $s=\pi/4+\frac12\arcsin z$ is the required solution.
Finally, \eqref{vtrec} and the analytic implicit-function theorem in Lemma \ref{gsep} produce positive analytic periodic warps.
\end{proof}

\subsection{A quantitative obstruction}

The moving walls may be pointwise nonempty even when no profile crosses a full period.
Indeed, a magnitude can reverse direction only at a critical time of $F$.
The next criterion turns this restriction into a direct integral obstruction.

\begin{proposition}[Periodic obstruction]\label{vobcritprop}
Fix $P>0$, $c\ne0$, $\varepsilon>0$, and a smooth $F>\sqrt\varepsilon(1+|c|)$ on $[0,P]$.
Choose $0<\alpha\le\beta<1$ so that both endpoint feasible intervals $[\sqrt{u_-(j)},\sqrt{u_+(j)}]$, $j\in\{0,P\}$, lie in $[\alpha,\beta]$.
Suppose $A:=\alpha-P/2>0$, $B:=\beta+P/2<1$, and set
$$
      g_F(t,r)=1-r^2-\frac{\varepsilon}{F(t)^2}(r^{-2}+c^2-1),
        \qquad
         \mu(t)=\min\{g_F(t,A),g_F(t,B)\}.
$$
If
\begin{equation}\label{vobcrit}
         \mu(t)>0\quad\text{whenever }0<t<P\text{ and }F'(t)=0,
       \qquad
     \int_0^P\sqrt{\max\{\mu(t),0\}}\,dt>\beta-\alpha,
\end{equation}
then \eqref{vscal} has no smooth solution on $[0,P]$.
\end{proposition}

\begin{proof}
Suppose that such a solution exists, and set $a=\cos s$.
Then $(a')^2=g_F(t,a)$ and $|a'|<1$.
The endpoint bounds imply $a(t)\in[A,B]$, while concavity in $r$ yields $g_F(t,a(t))\ge\mu(t)$.
At an interior zero of $a'$, we have $g_F(t,a)=0$, and differentiation yields
$
     0=\partial_tg_F(t,a)
       =\frac{2\varepsilon[1+(c^2-1)a^2]}{F^3a^2}F'.
$
Thus $F'=0$.
The first condition in \eqref{vobcrit} now gives $\mu(t)>0$, contradicting $0=g_F(t,a)\ge\mu(t)$.
Thus $a'$ has one sign, and
$$
          \beta-\alpha\ge |a(P)-a(0)|
         =\int_0^P|a'|\,dt
         \ge\int_0^P\sqrt{\max\{\mu(t),0\}}\,dt,
$$
contrary to our second condition.
\end{proof}

For example, when $c=\varepsilon=1$, $P=1/4$, and $F(t)=2.0001+1.9999\sin^2(4\pi t)$,
we can take $[\alpha,\beta]=[0.70,0.72]$, hence $A=0.575$, $B=0.845$.
Then $\mu(P/2)>0.19$ and $\mu>0.12$ on $[P/4,3P/4]$, so \eqref{vobcrit} holds.
Thus no smooth reconstruction exists for this parameter pair.

\end{document}